\documentclass[11pt]{amsart}

\usepackage{amssymb,amsfonts,amsmath,amsopn,amstext,amscd,latexsym,verbatim,hyperref}
\usepackage{amsthm,graphicx}

\newtheorem{Proposition}{Proposition}[section]
\newtheorem{Theorem}[Proposition]{Theorem}
\newtheorem{Corollary}[Proposition]{Corollary}
\newtheorem{Lemma}[Proposition]{Lemma}

\newtheorem{Remark}[Proposition]{Remark}

\newcommand{\bbA}{{\mathbb A}}
\newcommand{\bbB}{{\mathbb B}}
\newcommand{\bbC}{{\mathbb C}}

\newcommand{\bbN}{{\mathbb N}}

\newcommand{\bbP}{{\mathbb P}}
\newcommand{\bbQ}{{\mathbb Q}}

\newcommand{\bbZ}{{\mathbb Z}}

\newcommand{\frg}{{\mathfrak g}}

\newcommand{\frp}{{\mathfrak p}}

\newcommand{\frt}{{\mathfrak t}}

\newcommand{\cC}{{\mathcal C}}

\newcommand{\cF}{{\mathcal F}}

\newcommand{\cI}{{\mathcal I}}

\newcommand{\cO}{{\mathcal O}}

\newcommand{\cX}{{\mathcal X}}
\newcommand{\cY}{{\mathcal Y}}

\newcommand{\Qp}{\mathbb Q_p} 
\newcommand{\Cp}{\mathbb C_p} 

\newcommand{\N}{\mathbb N}
\newcommand{\Z}{\mathbb Z}
\newcommand{\R}{\mathbb R}
\renewcommand{\P}{\mathbb P}

\begin{document}

\title{The pro-\'etale cohomology of Drinfeld's upper half space}
\author{Sascha Orlik}
\address{Bergische Universit\"at Wuppertal, Fakult\"at f\"ur Mathematik und Naturwissenschaften, Fachgruppe Mathematik und Informatik, Gau\ss{}stra\ss{}e 20, D-42119 Wuppertal,  Germany}
\email{orlik@math.uni-wuppertal.de}
\date{}

\begin{abstract}
We determine the geometric pro-\'etale cohomology of Drinfeld's upper half space over a $p$-adic field  $K$.
The strategy is different from \cite{CDN} and is based on the approach developed in \cite{O} describing global sections of
equivariant vector bundles.
\end{abstract}

\maketitle

\normalsize


\section*{Introduction}

Let $K$ be a finite extension of $\Qp$ with absolute Galois group $\Gamma_K={\rm Gal}(\overline{K}/K)$ and let $\bbC_p$ be the completion of an algebraic closure of $K$.
We denote by 
$$ \cX=\mathbb \P_K^d\setminus \bigcup\nolimits_{H \varsubsetneq K^{d+1}}\mathbb \P(H)$$
(the complement of all  $K$-rational hyperplanes in projective space $\mathbb \P_K^d$) 
Drinfeld's upper half space \cite{D} of dimension $d\geq1$ over $K$.  This is a rigid analytic variety over $K$
which is equipped with a natural action of $G={\rm GL}_{d+1}(K)$. 
In \cite{CDN} Colmez, Dospinescu and Niziol determined the pro-\'etale cohomology of $\cX_{\bbC_p}$ as a special case 
considering more generally  Stein spaces $X$ which have a semistable weak formal scheme  over the ring of integers $O_K.$
It turns out that these  cohomology
groups are strictly exact extensions 
\begin{equation}\label{CDN_thm}
0 \rightarrow \Omega^{s-1}(\cX)/D^{s-1}(\cX)\hat{\otimes}_K \bbC_p \rightarrow H^s(\cX_{\bbC_p},\bbQ_p(s)) \rightarrow 
v^{G}_{{P_{(d-s+1,1,\ldots,1)}}}(\bbQ_p)' \rightarrow 0
\end{equation}
of $G\times \Gamma_K$-modules.
Here $\Omega^{s-1}$ is the sheaf of differential forms of degree $s-1$ on $\bbP_K^d$,  
$D^{s-1}=\ker(d^{s-1})$ where $d^{s-1}:\Omega^{s-1} \to \Omega^s$ is the differential morphism  and $v^{G}_{{P_{(d-s+1,1,\ldots,1)}}}(\bbQ_p)'$ is the (strong) topological dual of the 
generalized smooth Steinberg representation
attached to the decomposition $(d-s+1,1,\ldots,1)$ of $d+1.$
 From the above extensions one tells that these invariants are finer than the de Rham cohomology 
of $\cX$  determined by Schneider and Stuhler \cite{SS}. The latter objects are  essentially (replace $\bbQ_p$ by $K$) given  by the representations on the RHS  of the above sequences. 
For the proof of their result, Colmez, Dospinescu and Niziol  use syntomic and Hyodo-Kato cohomology, comparison isomorphisms
(Fontaine-Messing period morphisms) and moreover as already mentioned above a  ($p$-adic) semi-stable weak formal model of
$\cX$ over $O_K$. In the meantime Colmez and Niziol \cite{CN2} generalized their results to a large extent to arbitrary smooth
rigid analytic (dagger) varieties.

Nevertheless, our goal in this paper is to give an alternative approach for the determination  of the (p-adic) pro-\'etale cohomology
groups of $\cX_{\bbC_p}.$ The strategy is based  on the machinery developed in 
\cite{O} for describing global sections of equivariant vector bundles on $\cX.$ The advantage is that it reduces the computation of the pro-\'etale cohomology to 
simpler geometric objects, as open or punctured discs. For the latter aspect we follow the idea of Le Bras \cite{LB} who sketched
a general strategy for Stein spaces and carried this out for open and essentially for punctured discs.\footnote{This strategy 
has been also used recently by
Guido Bosco \cite{Bo} in the case of Drinfeld's upper half space.}
Moreover, we have to consider as a technical ingredient local cohomology groups 
$H^{\ast}_{\P^j_K(\epsilon)}(\P^d_K,\Omega^{\dag,s-1}/D^{\dag,s-1})$ with support in certain tubes  $\P^j_K(\epsilon)$
of projective subspace $\P^j_K$ and with coefficients in the dagger sheaf attached to $\Omega^{s-1}/D^{s-1}.$

\smallskip
Another feature of our approach is that we are able to make more precise the structure of the $G$-representation 
in the spirit 
of \cite{O,OS}. For any integer $j\geq 0$, let $P_{(j+1,d-j)}\subset G$ be the standard parabolic subgroup of $G$ attached to the decomposition 
$(j+1,d-j) $ of $d+1$ and $L_{(j+1,d-j)}$ its Levi factor. Our main theorem  is:

\medskip
{\it \noindent {\bf Theorem:} i)  For any integer $s\geq 0$, there is an extension\footnote{If the proof of 
Proposition \ref{prop_proetale_punctured_balls} allows us to show that the extensions mentioned there are strictly exact, 
then it follows that the extensions here are strictly exact, as well.}
\begin{equation*}
0 \rightarrow H^0(\cX,\Omega^{\dag,s-1}/D^{\dag,s-1})\hat{\otimes}_K \bbC_p \rightarrow H^s(\cX_{\bbC_p},\bbQ_p(s)) \rightarrow 
v^{G}_{{P_{(d-s+1,1,\ldots,1)}}}(\bbQ_p)' \rightarrow 0
\end{equation*}
of $G\times \Gamma_K$-modules.}

{\it ii) For any integer $s=1,\ldots,d$, there is a descending filtration $(Z^j)_{j=0,\ldots,d}$ on 
$Z^0=H^0(\cX,\Omega^{\dag,s-1}/D^{\dag,s-1})$ 
by closed subspaces together with isomorphisms of locally analytic $G$-representations
$$(Z^j/Z^{j+1})'  \cong \cF^G_{P_{(j+1,d-j)}}(H^{d-j}_{\bbP^{j}_K}(\bbP^d_K,\Omega^{s-1}/D^{s-1}), St_{d-j}), j=0,\ldots,d-1$$  
where $H^{d-j}_{\bbP^{j}_K}(\bbP^d_K,\Omega^{s-1}/D^{s-1})$ is the algebraic local Zariski cohomology and 
$St_{d-j}$ is the Steinberg representation of ${\rm GL}_{d-j}(K)$ considered as representation of $L_{(j+1,d-j)}.$

(Here we refer to \cite{OS} for the definition of the functors $\cF^G_{P_{(j+1,d-j)}}$).}

The content of this paper is organized as follows. The second part deals with the
study of the pro-\'etale local cohomology $H^{\ast}_{\P^j_{\bbC_p}(\epsilon)}(\P^d_{\bbC_p},\bbQ_p)$ with support
in the rigid analytic tubes $\P^j_{\bbC_p}(\epsilon).$  For doing so,  we have to analyze in the first part the analytic local cohomology groups 
$H^{\ast}_{\P^j_K(\epsilon)}(\P^d_K,\cF)$  where $\cF$ is one of the sheaves 
$D^{\dag,s-1}, \Omega^{\dag,s-1}/D^{\dag,s-1}$ from above. As for the latter objects, we cannot apply the same methods of \cite{O} 
directly, e.g. the \u{C}ech complex, since the sheaves are not coherent. We circumvent this problem by proving a vanishing 
result for the  higher cohomology groups of $\cF$ by using the machinery of van der Put \cite{vP} for overconvergent sheaves.
The local pro-\'etale cohomology groups  $H^{\ast}_{\P^j_{\bbC_p}(\epsilon)}(\P^d_{\bbC_p},\bbQ_p)$  are needed   in order to evaluate
the spectral sequence attached to some acyclic complex on the closed complement of $\cX_{\bbC_p}$ in $\P^d_{\bbC_p}$  in the final section.
This strategy was already used in \cite{O} for equivariant vector bundles. It works for the pro-\'etale cohomology, as well
since we can pull back this acyclic complex to the pro-\'etale site of $\P^d_{\bbC_p}$.
 Finally  we  show that $H^0(\cX,\Omega^{\dag,s-1}/D^{\dag,s-1})=\Omega^{\dag,s-1}(\cX)/D^{\dag,s-1}(\cX)$ for all
$s\geq 1.$  This result is needed for showing the compatibility of our result with formula (\ref{CDN_thm}) as our  approach 
gives rather rise to the space of global sections of the  sheaves $\Omega^{\dag,s-1}/D^{\dag,s-1}.$

\vspace{0.7cm}
{\it Notation:} We denote  by $p$ a prime,  by $K\supset \bbQ_p$ a finite extension of the field of $p$-adic integers $\Qp$, 
by  $O_K$ its ring of integers and by  $\pi$ a uniformizer of $K$. Let $|\; \;|: K \rightarrow \R$ be the normalized norm, i.e., $|\pi|= \#(O_K/(\pi))^{-1}.$
We denote by $\Cp$ the completion of an algebraic closure $\overline{K}$ of $K$ and 
extend the norm $| \; \;  |$ on it.  For a locally convex $K$-vector space $V$, we denote by $V'$ its strong dual, i.e., the $K$-vector space of continuous linear forms
equipped with the strong topology of bounded convergence.

We denote for a scheme $X$ (or rigid analytic variety) over $K$ by $X^{rig}$ (resp. $X^{ad}$) the  rigid analytic variety attached  (resp. adic space) 
to $X$. 
If $Y \subset \P^d_K$ is a closed algebraic $K$-subvariety and $\cF$ is a sheaf on $\P^d_K$ we write 
$H_Y^\ast(\P^d_K,\cF)$ for the corresponding local cohomology. If $Y$ is a rigid analytic subvariety 
(resp. pseudo-adic subspace) of $(\P^d_K)^{rig}$ (resp.  $(\P^d_K)^{ad}$)  we also write  
$H_Y^\ast(\P^d_K,\cF)$ instead of $H_Y^\ast((\P^d_K)^{rig},\cF^{rig})$ (resp. $H_Y^\ast((\P^d_K)^{ad},\cF^{ad})$) to simplify matters.  
For a scheme $X$ (or an adic space etc.) over $\bbC_p$, we denote by $H^i(X,\bbQ_p)$ the $p$-adic  pro-\'etale cohomology of it.

We use bold letters $\bf G, \bf P,\ldots $ to denote algebraic group schemes over $K$, whereas we use normal letters
$G,P, \ldots $ for their $K$-valued points of $p$-adic groups.
We use Gothic letters $\frg,\frp,\ldots $ for their Lie algebras. The corresponding enveloping algebras are denoted as usual by $U(\frg), U(\frp), \ldots
.$ Finally, we set ${\bf G}:= {\rm {\bf GL_{d+1}}}.$ 
Denote by ${\bf B\subset G}$ the Borel subgroup of lower triangular
matrices. Let ${\bf
T\subset G}$ be the diagonal torus. Let
$\Delta$ be the set of simple roots with respect to ${\bf T\subset B}$. For a
decomposition $(i_1,\ldots,i_r)$ of $d+1,$ let ${\bf
P_{(i_1,\ldots,i_r)}}$ be the corresponding standard-parabolic
subgroup of ${\bf G}$  and ${\bf
L_{(i_1,\ldots,i_r)}}$ its Levi component.
We  consider the algebraic
action $m: {\bf G} \times \P^d_K \rightarrow \P^d_K$ of ${\bf G}$ on
$\P_K^d$ given by
$$g\cdot [q_0:\cdots :q_d]:=m(g,[q_0:\cdots :q_d]):= [q_0:\cdots :q_d]g^{-1}.$$

\vspace{0.7cm}
{\it Acknowledgments:} I am very grateful to  Roland Huber for many helpful discussions. I wish to thank Gabriel Dospinescu, Wieslawa Niziol and Luc Illusie for their comments during a conversation 
at the Arithmetic Geometry conference in Bonn in October 2018 and Arthur-C\'esar Le Bras for pointing out to me the thesis \cite{Bo}
of Guido Bosco.  Finally I thank Michael Rapoport for his remarks on this paper. This research was conducted in the framework of the research training group
\emph{GRK 2240: Algebro-Geometric Methods in Algebra, Arithmetic and Topology},
which is funded by the DFG. 

\section{Some preparations}

We start to recall some further notations used in \cite{O}. Let $L$ be one of the fields $K$ or $\bbC_p.$
Let  $\epsilon \in \bigcup_{n\in \N}\sqrt[n]{|K^\times|} =
|\overline{K^\times}|$ be a $n$-th square root of
some absolute value in $|K^\times|$. For a closed
$L$-subvariety $Y \subset \mathbb P_L^d$, 
the open
$\epsilon$-neighborhood of $Y$ is defined by $$Y(\epsilon)=\Big\{z
\in \mathbb (\P_L^d)^{rig}\mid \mbox{ for any unimodular
representative }  \tilde{z} \mbox{ of $z,$  we have } $$
$$|f_j(\tilde{z})|\leq \epsilon
 \mbox{ for all } 1\leq j\leq r \Big\}.$$ 
Here $f_1,\ldots,f_r \in O_L[T_0,\ldots, T_d]$ are finitely many
homogeneous polynomials  generating the
vanishing ideal of $Y$. We suppose that each polynomial has at least
one coefficient in $O_L^\times.$ It is a
quasi-compact open rigid analytic subspace of $(\P^d_L)^{rig}.$ On
the other hand, the set
$$Y^-(\epsilon)=\Big\{z \in \mathbb (\P_L^d)^{rig}\mid \mbox{  for any unimodular representative } \tilde{z} \mbox{ of $z,$  we have } $$
$$| f_j(\tilde{z})|< \epsilon
 \mbox{ for all } 1\leq j\leq r \Big\}$$
 is the closed $\epsilon$-neighborhood of $Y.$ Again, it is an admissible open subset of $(\P^d_L)^{rig},$ but
 which is in general not quasi-compact.

Recall that Drinfeld's upper half space  $ \cX=\mathbb \P_K^d\setminus \bigcup\nolimits_{H \varsubsetneq K^{d+1}}\mathbb \P(H)$ is a rigid analytic Stein space over $K$ and its algebra of analytic functions $\cO(\cX)$
is a $K$-Fr\'echet space \cite{SS}. From this one deduces that for very vector bundle $\cF$ on $\bbP^d_K$ the space of global
sections $\cX$ has the structure of a $K$-Fr\'echet space. Since our $p$-adic group $G$ stabilizes $\cX$ we even get
for every ${\bf G}$-equivariant vector bundle $\cF$  on $\P^d_K$ the structure of a continuous $G$-representation on $\cF(\cX)$. 
Moreover its (strong) dual has the structure of a locally analytic $G$-representation in the sense of Schneider and Teitelbaum \cite{ST3}.

We consider the  de Rham complex of sheaves 
$$ \cO \stackrel{d_0}{\to} \Omega^1 \stackrel{d_1}{\to} \Omega^2 \to \ldots  \to \Omega^d $$
on the scheme $\P^d_K$  or rigid analytic variety $(\P^d_K)^{\rm rig}.$ Moreover, we let
$$ \cO^\dag \stackrel{d_0}{\to} \Omega^{\dag,1} \stackrel{d_1}{\to} \Omega^{\dag,2} \to \ldots  \to \Omega^{\dag,d} $$
be the de Rham complex on the dagger space $\P^{d,\dag}_K$ in the sense of Gro\ss{}e-Kl\"onne \cite{GK}.
In the following we use the notation $\Omega^{(\dag),s}$ etc. to handle both kind of sheaves simultaneously.
For convenience, we do not distinguish between $(\P^d_K)^{\rm rig}$ and $\P^{d,\dag}_K$ and also for other geometric objects since the underlying topological
spaces (Grothendieck topology) are the same. 
Both complexes are equivariant for the action of ${\bf G}$. 
Let $D^{(\dag),s}$ be the kernel of the sheaf homomorphism $d^s: \Omega^{(\dag),s} \to \Omega^{(\dag),s+1}$.
If $X \subset \P^{d,\dag}_K$ is as Stein space, then 
$\Omega^{\dag,s}(X)=\Omega^{s}(X).$ In particular we get $D^{\dag,s}(X)=D^{s}(X).$
We obtain ${\bf G}$-equivariant sheaves $D^{(\dag),s}$ and  $\Omega^{(\dag), s}/D^{(\dag),s}$ on $\bbP^d_K$, $s=0,\ldots,d.$ 
If we denote by $\cF^{(\dag)}$ one of them then there is an induced action of $G$ on the $K$-vector space of rigid analytic 
sections $\cF^{(\dag)}(\cX).$ Since $D^{(\dag),s}(\cX)$ is by the continuity of $d^s$ closed in $\Omega^{(\dag),s}(\cX)$ 
it follows from above that\footnote{Note that $\cF(\cX)=\cF^\dag(\cX)$}
$\cF^{(\dag)}(\cX)$ is a $K$-Fr\'echet space and its dual  $\cF^{(\dag)}(\cX)'$ is a locally
analytic $G$-representations, as well. 

\begin{Remark} \rm
From Proposition \ref{van_coh_drin} in the final section it will follow that the identity
$\Omega^{(\dag),s}(\cX)/D^{(\dag),s}(\cX) = H^0(\cX,\Omega^{\dag, s}/D^{\dag, s})$ is satisfied.
\end{Remark}

Fix an integer $0 \leq j \leq d-1.$
Let $$\P^j_K=V(T_{j+1},\ldots, T_{d})\subset \P^d_K$$ be the closed ${\bf P_{(j+1,d-j)}}$ stable $K$-subvariety
defined by the vanishing of the coordinates $T_{j+1},\ldots, T_{d}.$ 
The local cohomology groups $H^\ast_{\P^j}(\P^d_K,\cF)$ are then 
${\bf P_{(j+1,d-j)}} \ltimes U(\frg)$-modules, see \cite{O}. 
For any positive integer $n \in \N$, we consider the reduction map
\begin{equation}\label{reduction_map}
p_n: {\rm \bf GL_{d+1}}(O_K)\rightarrow {\rm \bf GL_{d+1}}(O_K/{(\pi^n)}).
\end{equation} 
Put
$P_{(j+1,d-j)}^n:= p_n^{-1}\big({\bf P_{(j+1,d-j)}}(O_K/{(\pi^n)})\big).$ 
This is a compact open subgroup of $G$ which stabilizes the $\epsilon_n$-neighborhood $\P^j_K(\epsilon_n)$ where $\epsilon_n:=|\pi|^n.$ 
Hence $H^\ast_{\P^j_K(\epsilon_n)}(\P^d_K,\cF^{(\dag)})$ is a $P_{(j+1,d-j)}^n \ltimes U(\frg)$-module.
Again as in the algebraic setting, we have
a long exact sequence of $P_{(j+1,d-j)}^n \ltimes U(\frg)$-modules
\begin{eqnarray*}
\cdots & \rightarrow &  H^{i}(\P^d_K,\cF^{(\dag)}) \rightarrow H^{i}(\P^d_K\setminus \P^j_K(\epsilon_n),\cF^{(\dag)}) 
\rightarrow H^{i}_{\P^j_K(\epsilon_n)}(\P^d_K,\cF^{(\dag)}) \\
& \rightarrow &   H^{i+1}(\P^d_K,\cF^{(\dag)}) \rightarrow \cdots.
\end{eqnarray*}

\medskip
In the following we study the analytic local cohomology groups  $H^\ast_{\P^j_K(\epsilon)}(\P^d_K,\cF^\dag)$
for 
$$\cF^\dag\in \Theta:=\{D^{\dag,s}\mid s=1,\ldots,d\} \cup \{\Omega^{\dag,s}/D^{\dag,s}\mid s=1,\ldots,d\}.$$  For any real numbers $0<\delta < \epsilon$, we consider the rigid analytic varieties
$$B(\epsilon):=\{z \in (\bbA_K^1)^{rig}\mid |z| \leq \epsilon \} \mbox{ and }C(\delta,\epsilon):=\{z \in  (\bbA_K^1)^{rig}\mid \delta \leq |z| \leq \epsilon \}$$ 
resp. 
$$B^-(\epsilon):=\{z \in  (\bbA_K^1)^{rig}\mid |z| < \epsilon \} \mbox{ and } C^-(\delta,\epsilon):=\{z \in  (\bbA_K^1)^{rig}\mid \delta < |z| < \epsilon \}.$$ 

\medskip
\begin{Proposition}\label{vanishing_D}
Let $X=\prod_{i=1}^d X_i\subset (\P^d_K)^{\rm rig}$ be some open subspace where for $i=1,\ldots, d$,  
$X_i\in \bigcup_{\delta >0, \epsilon >0} \{B(\epsilon), C(\delta,\epsilon)\},$
Then  $H^n(X,D^{\dag,s})=0$ and $H^n(X,\Omega^{\dag,s}/D^{\dag,s})=0$ for all $n>0.$
\end{Proposition}

\begin{proof}
Since $X$ is affinoid and $\Omega^{\dag,s}$ is coherent it suffices (by considering the long exact cohomology sequence attached 
to $0\to D^{\dag,s} \to \Omega^{\dag,s} \to \Omega^{\dag,s}/D^{\dag,s}\to 0)$ to prove the vanishing property 
for the sheaf $D^{\dag,s}$, cf. \cite[Prop.3.1]{GK2}.

 We follow here the machinery of van der Put \cite{vP}. The proof is by induction on $d$. 
 The case of the constant sheaf $D^{\dag,0}$ is treated in loc. cit. In particular this contributes to the start of the 
 induction, i.e., for $d=1$. The sheaf $D^{\dag,1}$ coincides in this case  with the coherent sheaf $\Omega^{\dag,1}$ whose 
 higher cohomology vanishes anyway, cf. \cite[Prop.3.1]{GK2}.
 
 Let $d>1$ and $X'=\prod_{i=1}^{d-1} X_i$.  We consider the projection map $\phi:X \to X'$ forgetting the last entry and
  the induced Leray spectral sequence $H^j(X',R^i\phi_\ast D^{\dag,s}) \Rightarrow H^{i+j}(X,D^{\dag,s}).$
 By the very definition the sheaves $D^{\dag,s}$ are  constructible. 
 Hence for any closed geometric point $z$ of $X'$
 there is by Theorem 2.3 of loc.cit. an isomorphism $H^i(\hat{Z},i^{-1} D^{\dag,s})=(R^i\phi_\ast D^{\dag,s})_z$ where
 $i: \hat{Z} \hookrightarrow \hat{X}$ is the inclusion map of the fiber\footnote{Here we adopt the notation of \cite{vP} to denote by $\hat{X}$ the space of closed geometric points of $X$} at $z$. 
 We shall prove that $(R^i\phi_\ast D^{\dag,s})_z=0$ and that  $H^i(X,\phi_\ast D^{\dag,s})=0$ for all $i>0$. We start with 
 the latter  aspect.  
 
 The complex  $\phi_\ast \Omega_X^{\dag,\bullet}$ is the (dagger) tensor product of the de Rham complex $\Omega_{X'}^{\dag,\bullet}$ on 
 $X'$ with the constant de Rham complex $\cO^\dag(X_d) \to \cO^\dag(X_d)dT_d$ of global sections on $X_d$.  Let 
$D_{X'}^{\dag,s-1}$ be the kernel of the morphism $d_{s-1}: \Omega^{\dag,s-1}_{X'} \to  \Omega^{\dag,s}_{X'}.$
 One verifies that $\phi_\ast D^{\dag,s}$ is the sum of the sheaves
 $D_{X'}^{\dag,s-1} \otimes^\dag \cO^\dag(X_d)dT_d$, $D_{X'}^{\dag,s}\otimes^\dag K$, and  some sheaf isomorphic to 
 $\Omega_{X'}^{\dag,s-1} \otimes^\dag \cO^\dag(X_d)$ 
 mod $D^{\dag,s-1}_{X'} \otimes^\dag K$. Here we use the strictness of the differentials as in the proof
 of \cite[Thm. 4.12]{GK2}. The higher cohomology groups 
 of the first two summands vanish by induction and the flatness of $\cO^\dag(X_d)dT_d$ over $K.$ Concerning the vanishing 
 of the 
 latter summand this follows from the fact that $X'$ is affinoid,  $\Omega_{X'}^{s-1}$ is coherent and again by induction and 
 flatness. In the same way, one checks that this vanishing property holds for possible intersections of these sheaves.

 Now we prove that $H^i(\hat{Z},i^{-1} D^{\dag,s})=0$ for all $i>0.$  We start with the observation that for any  
 admissible open subset $V \subset \hat{X_d}$ and any open admissible subset $U\subset \hat{X}$ with $\{z\}\times V\subset U$ 
 there is some open neighborhood $W$ of $z$ such that $W \times V \subset U$ \footnote{This is clear for classical points, i.e. for the rigid varieties $X$ without $\hat{}$. It transfers by the very definition of the topology
 to the enriched rigid varieties $\hat{X}$.}. From this topological fact, 
 we deduce by the very definition
 of the functor $i^{-1}$ that the pull-back $i^{-1} \Omega_X^\bullet$ is the (dagger) tensor product of the de 
 Rham complex $\cO^\dag_{X_d} \to \Omega^{\dag,1}_{X_d}$ on $X_d$ and the complex 
 $\Omega^{\dag,\bullet}_{X',z}$ given by the localisation of $\Omega^{\dag,\bullet}_{X'}$ in $z.$ Since the 
 functor $i^{-1}$ is exact we deduce as above that $i^{-1} D^{\dag,s}$ is
 the sum of the sheaves $(D_{X'}^{\dag, s})_z \otimes^\dag D^{\dag,0}_{X_d}$, $(D_{X'}^{\dag,s-1})_z \otimes^\dag \Omega^{\dag,1}_{X_d}$ and
 and some sheaf isomorphic to $\Omega^{\dag,s-1}_{X',z} \otimes^\dag \cO_{X_d}^{\dag}/(D_{X'}^{\dag,s-1})_z \otimes^\dag D^{\dag,0}_{X_d}$.
 Again by (even the start of) induction and flatness
we conclude the claim.
 \end{proof}

\begin{Corollary}\label{again_vanishing_D}
Let $X^-=\prod_{i=1}^d X_i^-\subset (\P^d_K)^{rig}$ be some open subspace where  for $i=1,\ldots, d$,
$X_i^-\in  \bigcup_{\delta >0, \epsilon >0} \{B^-(\epsilon), C^-(\delta,\epsilon)\}$. Then  
$H^n(X^-,D^{\dag,s})=0$ and $H^n(X^-,\Omega^{\dag,s}/D^{\dag,s})$ $=0$ for all $n>0.$
\end{Corollary}

\begin{proof}
As above (by considering the corresponding long exact cohomology sequence) it is enough to prove  the statement for the sheaf $D^{\dag,s}$.
 We start with the observation that $X^-$ is a Stein space for which an admissible affinoid covering $X^-=\bigcup_{k\in \bbN} U_k $ with affinoid objects 
 $U_k$ as before exists.  By the same reasoning as in  \S 2 of \cite{SS}  we have short exact sequences
$$0\rightarrow \varprojlim_k\nolimits^{(1)} H^{i-1}(U_k,D^{\dag,s}) \rightarrow H^i(X^-,D^{\dag,s}) \rightarrow \varprojlim_k H^i(U_k,D^{\dag,s})\rightarrow 0.$$
Thus we get the claim for $i\geq 2$ by applying the previous proposition. For $i=1$, we need to show that  $\varprojlim_k\nolimits^{(1)} H^{0}(U_k,D^{\dag,s}) =0.$ But $\varprojlim_k\nolimits^{(1)} H^{0}(U_k,D^{s}) =0$
as the projective system $(H^{0}(U_k,D^{s}))_k$ consists of Banach spaces where the transition maps have dense image so that the topological Mittag Leffler property is satisfied.
Thus it is enough to show that $\varprojlim_k H^{0}(U_k,D^s/D^{\dag,s}) =0.$ But by the definition of the sheaves $D^{\dag,s}$ the transition maps in the projective system
$H^{0}(U_k,D^s/D^{\dag,s}) $ are all zero.  Hence we see that  $\varprojlim_k H^{0}(U_k,D^s/D^{\dag,s}) =0.$
\end{proof}

\begin{Corollary}\label{vanishing_D_Y} 
Let $Y$ be one of the rigid analytic varieties  $X$ or $X^-$ considered above. Then $\Omega^{\dag,s}(Y)/D^{\dag,s}(Y) = 
H^0(Y,\Omega^{\dag,s}/D^{\dag,s}).$ 
\end{Corollary}

\begin{proof}
 This follows from the corresponding long exact cohomology sequence and the vanishing of the first cohomology $H^1(Y,D^{\dag,s})$.
\end{proof}

As a first application we may deduce:                   
                  
\begin{Lemma}\label{vanishing_OmodD} Let $s\geq 1$. Then  $H^i(\bbP^d_K,\Omega^{\dag,s-1}/D^{\dag,s-1})=0$ for $i \geq 0$ .

\end{Lemma}

\begin{proof}
We consider the standard covering $\bbP^d_K=\bigcup_{k=0}^d D_+(T_k)_{1}$ by closed balls, i.e,   where  
$D_+(T_k)_{1}=\{x\in \P^d_k \mid |x_k| \geq |x_j|\;\; \forall j \neq k\}$. 
 By applying Proposition \ref{vanishing_D} to the open subvarieties $D_+(T_k)_{1}, k=0,\ldots,d$ and its intersections
 we see that the corresponding \v{C}ech complex with values in $D^{\dag,s-1}$ computes
 $H^i(\bbP^d_K,D^{\dag,s-1}).$
 Now the proof is by induction on $s$. For $s=1$, we consider the short exact sequence
 $$0\to D^{\dag,0} \to \cO^\dag \to \cO^\dag/D^{\dag,0} \to 0. $$ 
  Using the fact that $H^0(X,D^{\dag,0})=K$ for any rigid analytic variety $X$ appearing in the \v{C}ech complex induced by
 the above covering we see by a combinatorial argument that $H^i(\bbP^d_K,D^{\dag,0})=0$ for $i > 0.$ 
 As for $\cO^\dag$ we have $H^i(\bbP^d_K,\cO^\dag)=H^i(\bbP^d_K,\cO)=0$ for any $i>0$ 
 \cite[Thm 5.1]{Ha}, \cite[Thm. 3.2]{GK2}. For $i=0$, the map $H^0(\bbP^d_K,D^{\dag,0}) \to H^0(\bbP^d_K,\cO^\dag)=K$ is clearly an isomorphism and the
 start of induction is shown.
 
 Now let $s> 1.$ For any  $X$ appearing as geometric object in the \u{C}ech complex,  there are by using
 Corollary \ref{vanishing_D_Y} and since $X$ is smooth short exact sequences
 $$0 \to (\Omega^{\dag,s-1}/D^{\dag,s-1})(X) \to D^{\dag,s}(X) \to H^s_{dR}(X) \to 0 $$
 and 
 \begin{equation}\label{exact_sequence}
  0\to D^{\dag,s} \to \Omega^{\dag,s} \to \Omega^{\dag,s}/D^{\dag,s} \to 0. 
 \end{equation}
  By induction hypothesis and by reconsidering the above covering we get from the first exact sequence 
 isomorphisms $H^i(\bbP^d_K,D^{\dag,s})=H^i(H^s_{dR}(\cdot)), i\geq 0$ where $H^s_{dR}(\cdot)$ is the complex
\begin{multline}
\bigoplus_{0\leq k \leq d} H^s_{dR}(D_+(T_k)_{1})\rightarrow \!\!\!\!\!
\bigoplus_{0\leq k_1<k_2 \leq d} H^s_{dR}(D_+(T_{k_1})_{1} \cap D_+(T_{k_2})_{1})
\rightarrow \cdots  \\ \cdots \rightarrow H^s_{dR}(D_+(T_{0})_{1}\cap \cdots \cap D_+(T_d)_{1}).
\end{multline}

 Now again by \cite[III, Exercise 7.3]{Ha} $H^i(\bbP^d_K,\Omega^{\dag,s}) = H^i(\bbP^d_K,\Omega^s) \neq 0$  
 iff $i=s$ which is moreover then a one-dimensional 
 $K$-vector space.  This is exactly induced by $H^i(H^s_{dR}(\cdot))$ and all other
groups $H^i(H^s_{dR}(\cdot))$ vanish.  Thus we deduce that $H^i(\bbP^d_K,\Omega^{\dag,s}/D^{\dag,s})=0$ for all $i\geq 0.$
 \end{proof}

 As a byproduct we see that $D^{\dag,s}$ has the same cohomology on $\P^d$ as $\Omega^s,$ i.e. 
 \begin{equation}\label{same_coh}
  H^n(\P^d_K,D^{\dag,s})=H^n(\P^d_K,\Omega^s)
 \end{equation}
 for all $n\geq 0$

As for the next application, we consider  for any integer $0\leq j \leq d-1$, the complement 
$\P^d_K\setminus \P^j_K(\epsilon)$ of the tube $\P^j_K(\epsilon)$ in projective space. 
One checks that there is a covering
$$\P^d_K\setminus \P^j_K(\epsilon)=\bigcup_{k=j+1}^d V(k;\epsilon)$$
where 
 $$V(j+1;\epsilon)= 
\Big\{[x_0:\ldots:x_d] \in \P^d_K \mid\, |x_{j+1}| > |x_l|\cdot \epsilon \;\;  \forall \, l  < j+1 \Big\}$$
and
 $$V(k;\epsilon)= 
\Big\{[x_0:\ldots:x_d] \in \P^d_K \mid\,  |x_{k}| > |x_l|\cdot \epsilon \;\;  \forall \, l  < j+1 ,
 |x_{k}| > |x_l| \;\;  \forall \, j+1 \leq l < k \Big\}$$
for $k >j+1.$  These are admissible open subsets of $(\bbP_K^d)^{rig}$ which are even Stein spaces. The same holds true for arbitrary intersections
of them. More concretely, there is the following description.

\begin{Lemma}\label{lemma_V}
Let $I=\{i_1< \cdots < i_r\} \subset \{j+1,\ldots,d\}.$ Then 
$$ \bigcap_{k\in I} V(k;\epsilon) \cong  B^-(1/\epsilon)^{j+1} \times B^-(1)^{i_1-1-j} \times \prod_{k=2}^{r} \big(
C^- (0,1)  \times B^- (1)^{i_{k} - i_{k-1}-1} \big) 
\times \bbA^{d - i_r}_K$$
\end{Lemma}

\begin{proof}
We consider the map $\bigcap_{k\in I} V(k;\epsilon) \to \bbA^d $ defined by 
$[x_0:\cdots:x_d] \to (y_0,\ldots,y_{d-1})$ where $y_i=x_i/x_{i_1}$ for $i<i_1$, 
$y_i=x_i/x_{i_2}$ for $i_1 \leq i < i_2$,  $y_i=x_i/x_{i_3}$ for $i_2 \leq i < i_3$, \ldots, 
$y_i=x_i/x_{i_r}$ for $i_{r-1} \leq i < i_r$ and $y_i=x_i/x_{i_r}$ for $i > i_r.$
The image is contained in the RHS of the stated isomorphisms. We define an inverse morphism by
$(y_0,\ldots,y_{d-1}) \mapsto  [x_0:\cdots:x_d]$ with 
$x_i=y_i, i<i_1$, $x_{i_1}=1$, $x_i=y_iy_{i_1}^{-1}, i_1< i < i_2$, $x_{i_2}=y_{i_1}^{-1}$, $x_i=y_iy_{i_1}^{-1}y_{i_2}^{-1}, 
i_2 < i < i_3$, $x_{i_3}=y_{i_1}^{-1}y_{i_2}^{-1}, $ $x_i=y_i y_{i_1}^{-1}y_{i_2}^{-1}y_{i_3}^{-1}, i_3 < i < i_4,$ 
$x_{i_3}=y_{i_1}^{-1}y_{i_2}^{-1}y_{i_3}^{-1}$, \ldots, $x_i=y_i y_{i_1}^{-1}y_{i_2}^{-1}\cdots y_{i_r}^{-1}, i > i_r.$ 
\end{proof}

By Lemma \ref{lemma_V} and Corollary \ref{again_vanishing_D} we may compute the cohomology 
$H^{\ast}(\P^d_K\setminus \P^j_K(\epsilon_n),\cF^\dag)$ for 
 $\cF^\dag\in \Theta$ via the  \u{C}ech complex
 \begin{multline}\label{cechcomplex}
C_n^\bullet \cF^\dag : \bigoplus_{j+1\leq k \leq d} \cF^\dag(V(k;\epsilon_n)) \rightarrow \!\!\!\!\!\bigoplus_{j+1 \leq k_1<k_2 \leq d} 
\cF^\dag(V(k_1;\epsilon_n) \cap V(k_2;\epsilon_n)) \rightarrow \cdots  \\ \cdots 
\rightarrow \cF^\dag(V(j+1;\epsilon_n) \cap \cdots \cap V(d;\epsilon_n)) .
\end{multline}

\smallskip
\begin{Remark}\label{rmk_17}
In \cite{O} we proved via this approach that the cohomology groups
$H^\ast(\P^d_K\setminus \P^j_K(\epsilon_n),\Omega^s)=H^\ast(\P^d_K\setminus \P^j_K(\epsilon_n),\Omega^{\dag,s})$ 
and $H^\ast_{\P^j_K(\epsilon_n)}(\P^d_K,\Omega^s)=H^\ast_{\P^j_K(\epsilon_n)}(\P^d_K,\Omega^{\dag,s})$  are $K$-Fr\'echet spaces with  the structure of a continuous  
$P_{(j+1,d-j)}^n \ltimes U(\frg)$-module in which the algebraic cohomology $H^\ast(\P^d_K \setminus \P^j_K,\Omega ^s)$ resp.  
$H^\ast_{\P^j_K}(\P^d_K,\Omega^s)$   is a dense subspace.
Since the differential maps $d^s:\Omega^s(U) \to \Omega^{s+1}(U)$  are continuous for any open subvariety $U$ appearing 
in the  complex (\ref{cechcomplex}) and  $H^{\ast}(\P^d_K\setminus \P^j_K(\epsilon_n),D^{\dag,s}) 
\cap H^{\ast}(\P^d_K\setminus \P^j_K,\Omega^s)=H^{\ast}(\P^d_K\setminus \P^j_K,D^s)$ (cf. Prop. \ref{letzte_prop_1})  one checks now easily that the same is satisfied for the 
modules $H^{\ast}(\P^d_K\setminus \P^j_K(\epsilon_n),\cF^\dag)$ and $H^\ast_{\P^j_K(\epsilon_n)}(\P^d_K,\cF^\dag).$
This aspect can be precised as follows. For $j<  k \leq d$ and $\epsilon > \epsilon_n$, consider the 
open affinoid subvarietes
$$\bar{V}(j+1;\epsilon)= 
\Big\{[x_0:\ldots:x_d] \in \P^d_K \mid\, |x_{j+1}| \geq |x_l|\cdot \epsilon \;\;  \forall \, l  < j+1 \Big\}$$
and
 $$\bar{V}(k;\epsilon)= 
\Big\{[x_0:\ldots:x_d] \in \P^d_K \mid\,  |x_{k}| \geq |x_l|\cdot \epsilon \;\;  \forall \, l  < j+1 ,
 |x_{k}| \geq |x_l| \;\;  \forall \, j+1 \leq l < k \Big\}$$
and form the attached  \u{C}ech complex
 \begin{multline}\label{cechcomplex_b}
C_\epsilon^\bullet \cF : \bigoplus_{j+1\leq k \leq d} \cF(\bar{V}(k,\epsilon)) \rightarrow \!\!\!\!\!\bigoplus_{j+1 \leq k_1<k_2 \leq d} 
\cF(\bar{V}(k_1,\epsilon)\cap \bar{V}(k_2,\epsilon)) \rightarrow \cdots  \\ \cdots 
\rightarrow \cF(\bar{V}(j+1,\epsilon) \cap \cdots \cap \bar{V}(d,\epsilon)).
\end{multline}
If we denote by $U_\epsilon$ some open affinoid subvariety appearing in this complex (and similarly $U_{\epsilon'}$ for 
$\epsilon' < \epsilon$), 
then the restriction maps 
$\Omega^s(U_\epsilon) \to \Omega^s(U_{\epsilon'})$ of Banach spaces are 
injective, continuous and 
have dense image. It follows that the same holds true for the maps $D^s(U_\epsilon) \to D^s(U_{\epsilon'})$ of 
Banach spaces (!). As the
the map $d^s$ is even strict, we deduce that the the functor $\varprojlim_{\epsilon \to \epsilon_n}$ is exact with respect to 
the complexes  $C^\bullet_\epsilon \cF$. Since we additionally  have 
$\varprojlim_{\epsilon \to \epsilon_n} \cF(U_\epsilon)= 
\varprojlim_{\epsilon \to \epsilon_n} \cF^\dag(U_\epsilon)=\cF^\dag(U_{\epsilon_n})$   we may identify 
$H^i(\P^d_K\setminus \P^j_K(\epsilon),\cF^\dag)$ with the projective limit 
of $K$-Banach spaces $\varprojlim_{\epsilon \to \epsilon_n} H^i(C^\bullet_\epsilon \cF).$ 
\end{Remark}


Next we observe that     
  \begin{equation}\label{van_local_coh}
  H^n_{\P^j_K(\epsilon)}(\P^d_K,D^{\dag,s})=0 \,\,\,\ \forall n> d-j
 \end{equation}
 by the length of the \u{C}ech complex (\ref{cechcomplex}).
The following result is known for coherent sheaves  (by the smoothness of $\P^j_K$), cf. \cite{O}.

\begin{Lemma}\label{coh_vanishing}
Let $0\leq j \leq d-1$. Then the cohomology groups $H^i_{\P^j_K(\epsilon_n)}(\P^d_K,\Omega^{\dag,s}/D^{\dag,s})$ and 
$H^i_{\P^j_K(\epsilon_n)}(\P^d_K,D^{\dag,s})$ 
vanish for $i<d-j.$
\end{Lemma}

\begin{proof}
The case $j=d-1$ is trivial by Lemma \ref{vanishing_OmodD} and since $K=H^0(\P^d_K,D^{\dag,0}) = 
H^0(\P^d_K\setminus \P^j_K(\epsilon_n),D^{\dag,0})$  . So let $j<d-1$.

The proof is similar to Lemma \ref{lemma_V}. By Lemma \ref{vanishing_OmodD} we need to show that 
$H^i(\P^d_K\setminus \P^j_K(\epsilon_n),\Omega^{\dag,s}/D^{\dag,s})=0$ for $i<d-j-1.$ We consider the covering of $\P^d_K\setminus
\P^j_K(\epsilon_n) = \bigcup_{k>j} V(k;\epsilon_n).$ For $s=0$, we have again 
$H^i(\P^d_K\setminus \P^j_K(\epsilon_n),D^{\dag,0})=0$ 
for all $i>0$ and $H^0(\P^d_K\setminus \P^j_K(\epsilon_n),D^{\dag,0})=K.$
Since the stated vanishing is true for coherent sheaves we see that
$H^i(\P^d_K\setminus \P^j_K(\epsilon_n),\cO^\dag)=H^i(\P^d_K,\cO^\dag)$ for all such $i$ and the result follows.

For $s>0$ we reconsider the exact sequence (\ref{exact_sequence}). By induction hypothesis the statement is true for the 
sheaf $\Omega^{\dag,s-1}/D^{\dag,s-1}$. Hence $H^i(\P^d_K\setminus \P^j_K(\epsilon_n),D^{\dag,s})=H^i(H^s_{dR}(\cdot))$ for all $i<d-j-1$ 
where $H^s_{dR}(\cdot)$ is defined similar as before with respect to the above covering of 
$\P^d_K\setminus \P^j_K(\epsilon_n).$
  The latter term coincides
with $H^i(\P^d_K,\Omega^{\dag,s})$ which is thus $H^i(\P^d_K\setminus \P^j_K(\epsilon_n),\Omega^{\dag,s})$ 
 since $i<d-j-1$. As before the result follows.
 \end{proof}
 
 \medskip

\begin{Proposition}\label{letzte_prop_1}
There are for all $s\geq 1$,  strictly   exact sequences 
$$0 \to  H^{d-j}_{\P^j_K(\epsilon_n)}(\P^d_K,D^{\dag,s-1}) \to H^{d-j}_{\P^j_K(\epsilon_n)}(\P^d_K,\Omega^{\dag,s-1}) \to 
H^{d-j}_{\P^j_K(\epsilon_n)}(\P^d_K,\Omega^{\dag,s-1}/D^{\dag,s-1}) \to 0.$$
\end{Proposition}

\begin{proof}
By Lemma \ref{coh_vanishing} the above sequence is exact from the left. As for the right exactness
we claim that the map $H^{d-j+1}_{\P^j_K(\epsilon_n)}(\P^d_K,D^{\dag,s-1}) \to H^{d-j+1}_{\P^j_K(\epsilon_n)}(\P^d_K,\Omega^{\dag,s-1})$
is an isomorphism. For verifying this, we consider
the long exact sequence 
$$\ldots \to  H^{d-j}(\P^d_K\setminus {\P^j_K(\epsilon_n)},D^{\dag,s-1}) \to H^{d-j+1}_{\P^j_K(\epsilon_n)}(\P^d_K,D^{\dag,s-1}) 
\to H^{d-j+1}(\P^d_K,D^{\dag,s-1}) \to \ldots$$
The term $H^{d-j}(\P^d_K\setminus {\P^j_K(\epsilon_n)},D^{\dag,s-1})$ vanishes by the length of the \u{C}ech complex.
Further $H^{d-j+1}(\P^d_K,D^{\dag,s-1})=H^{d-j+1}(\P^d_K,\Omega^{\dag,s-1}) $ by identity (\ref{same_coh}) and 
$H^{d-j+1}(\P^d_K,\Omega^{s-1})= 
H^{d-j+1}_{\bbP^j(\epsilon_n)}(\P^d_K,\Omega^{\dag,s-1})$ by the same reasoning. Since the vector spaces in question are at 
most one-dimensional the  claim follows.

As for the topological statement  this follows easily from
the fact that the differential map $d^{s-1}$ is continuous and therefore $D^{\dag,s-1}(X)$ is closed in 
$\Omega^{\dag,s-1}(X)$ for every rigid analytic variety $X$ appearing in the \u{C}ech complex. 
\end{proof}

\begin{Remark}\label{remark_cato} \rm
The sequence
$$0 \to  H^{d-j}_{\P^j_K}(\P^d_K,D^{\dag,s-1}) \to H^{d-j}_{\P^j_K}(\P^d_K,\Omega^{\dag,s-1}) \to 
H^{d-j}_{\P^j_K}(\P^d_K,\Omega^{\dag,s-1}/D^{\dag,s-1}) \to 0$$
consisting of analytic local cohomology groups is exact, as well. This follows by taking the projective limit 
$\varprojlim_n$ and applying the topological Mittag-Leffler criterion.
By density this fact is also true for the corresponding sequence of algebraic local cohomology groups of schemes  
$$0 \to  H^{d-j}_{\P^j_K}(\P^d_K,D^{s-1}) \to H^{d-j}_{\P^j_K}(\P^d_K,\Omega^{s-1}) \to H^{d-j}_{\P^j_K}(\P^d_K,\Omega^{s-1}/D^{s-1}) \to 0.$$
In fact this sequence is the same as the pull back of the above sequence to $H^{d-j}_{\P^j_K}(\P^d_K,\Omega^{s-1}).$
\end{Remark}

\medskip
\begin{Remark}
\rm Since  the (strong) topological dual $H^{d-j}_{\P^j_K(\epsilon_n)}(\P^d_K,\Omega^{s-1})'$ is a locally analytic 
$P_{(j+1,d-j)}^n$-representation by \cite[Cor. 1.3.9]{O}, the same is true for 
$H^{d-j}_{\P^j_K(\epsilon_n)}(\P^d_K,\Omega^{s-1}/D^{s-1})'$
as a closed $P_{(j+1,d-j)}^n$-stable subspace. 
\end{Remark}

\medskip
\section{Some local pro-\'etale cohomology groups}

In the sequel we denote for a rigid analytic variety $X$ over $K$ by $X_{\bbC_p}$ its base change to $\bbC_p.$ 
We shall determine  in this section the local pro-\'etale cohomology groups $H^\ast_{\P^j_{\bbC_p}(\epsilon_n)}(\P^d_{\bbC_p},\bbQ_p).$

As usual there is a long exact cohomology sequence
\begin{eqnarray}\label{long_exact}
\cdots & \rightarrow & H^{i-1}(\P^d_{\bbC_p}\setminus \P^j_{\bbC_p}(\epsilon_n),\bbQ_p) \rightarrow 
H^i_{\P^j_{\bbC_p}(\epsilon_n)}(\P^d_{\bbC_p},\bbQ_p) \rightarrow H^i(\P^d_{\bbC_p},\bbQ_p) \\ \nonumber
& \rightarrow & H^i(\P^d_{\bbC_p} \setminus \P^j_{\bbC_p}(\epsilon_n),\bbQ_p) \rightarrow \cdots .
\end{eqnarray}
As for the computation of $H^i(\P^d_{\bbC_p}\setminus \P^j_{\bbC_p}(\epsilon_n), \bbQ_p)$, we consider the 
spectral sequence
$$ E_1^{p,q} \Rightarrow E^{p+q}=H^{p+q}(\P^d_{\bbC_p}\setminus \P^j_{\bbC_p}(\epsilon_n), \bbQ_p)$$
with respect to the covering of Stein spaces 
$$\P^d_{\bbC_p}\setminus \P^j_{\bbC_p}(\epsilon_n)=\bigcup_{k=j+1}^d V(k;\epsilon_n)_{\bbC_p}.$$ 
The line $E^{\bullet,s}$ is given by the complex
\begin{multline}\label{Cechanal}
\bigoplus_{j+1\leq k \leq d} H^s((V(k;\epsilon_n)_{\bbC_p},\bbQ_p) \rightarrow \!\!\!\!\!
\bigoplus_{j+1 \leq k_1<k_2 \leq d} H^s(V(k_1;\epsilon_n)_{\bbC_p}\cap V(k_2;\epsilon_n)_{\bbC_p},\bbQ_p) 
\rightarrow \\ \cdots \rightarrow H^s(V(j+1;\epsilon_n)_{\bbC_p}\cap \cdots \cap V(d;\epsilon_n)_{\bbC_p},\bbQ_p).
\end{multline}

Let $U=V(k_1;\epsilon_n) \cap V(k_2;\epsilon_n) \cap \ldots \cap  V(k_r;\epsilon_n)$ be 
some intersection of these Stein spaces which appear in the above complex. Then this is a Stein space, as well, and the 
geometric pro-\'etale  cohomology has the following description.
\begin{Proposition}\label{prop_proetale_punctured_balls}
For $s\geq 0$, there is an extension 
$$0 \to \Omega^{s-1}(U)/D^{s-1}(U)\hat{\otimes}_K \bbC_p(-s)  \to H^s(U_{\bbC_p},\bbQ_p)\to H^s_{dR}(U,\bbQ_p)(-s) \to 0 $$ 
 where $H^s_{dR}(U,\bbQ_p)=\wedge^s (\bbQ_p^{r-1})$ is a $\bbQ_p$-vector space\footnote{ which can be expressed via  
 Hyodo-Kato cohomology} with $H^s_{dR}(U,\bbQ_p)\otimes_{\bbQ_p} K=H^s_{dR}(X).$
 \end{Proposition}
 
 \begin{proof}
  The proof is essentially contained in \cite{LB} as we shall explain.  At first we consider the short exact sequence of 
  pro-\'etale sheaves 
  \begin{equation}\label{ses}
  0 \to \bbQ_p \to \bbB[1/t]^\varphi \to \bbB_{dR}/\bbB^+_{dR}\to 0 
  \end{equation}
  on $U_{\bbC_p}$
and determine the cohomology of the period sheaves. As for latter sheaf, we deduce by
entering the proof of \cite[Prop. 3.17]{LB} that for $k>d$ the expression $H^i(U_{\bbC_p},\bbB^+_{dR}/t^k)$ is a successive extension
$$0 - {\rm coker}(d_i)(k-i-1) - H^i_{dR}(U_{\bbC_p},\bbC_p)(k-i-2) - \cdots - H^i_{dR}(U_{\bbC_p},\bbC_p)(-i-1) - {\rm ker}(d_i)(-i) - 0$$  
  (the terms in the middle are all Tate twists of $H^i_{dR}(U_{\bbC_p},\bbC_p)$).
 Now we write ${\rm ker}(d_i)(-i)$ as a (split) extension 
 $0 \to \Omega^{i-1}(U_{\bbC_p})/D^{i-1}(U_{\bbC_p})(-i) \to {\rm ker}(d_i)(-i) \to H^i_{dR}(U_{\bbC_p},\bbC_p)(-i) \to 0$ so 
 that the above expression can be rewritten as a successive extension
 as
 $$0 - {\rm coker}(d_i)(k-i-1) - H^i_{dR}(U_{\bbC_p},\bbC_p)(k-i-2) - \cdots - H^i_{dR}(U_{\bbC_p},\bbC_p)(-i) - \Theta_i - 0$$ with
 $$\Theta_i:=\Omega^{i-1}(U_{\bbC_p})/D^{i-1}(U_{\bbC_p})(-i). $$
 Passing to the limit as $k \to \infty$ we get an extension
 $$ 0 \to  H^i_{dR}(U_{\bbC_p},\bbC_p)(-i) \otimes B^+_{dR} \to  H^i(U_{\bbC_p},\bbB^+_{dR}) \to \Theta_i \to 0.$$
 
As for the sheaf $\bbB_{dR}$ we follow the reasoning in \cite[Prop. 3.19]{LB} to conclude that
$$H^i(U_{\bbC_p},\bbB_{dR})= H^i_{dR}(U_{\bbC_p},\bbC_p)(-i) \otimes B_{dR}.$$ Hence we see via the long exact cohomology sequence that $H^i_{dR}(U_{\bbC_p},\bbB_{dR}/\bbB^+_{dR})$ is an extension
 $$ 0 \to  H^i_{dR}(U_{\bbC_p},\bbC_p)(-i) \otimes B_{dR}/B^+_{dR} \to  H^i(U_{\bbC_p},\bbB_{dR}/\bbB^+_{dR}) \to \Theta_{i+1} \to 0.$$
 
Concerning the sheaf $\bbB[1/t]^\varphi$ we use Lemma \ref{lemma_V}, \cite[Cor. 3.31]{LB} and  the proof of \cite[Prop. 3.22]{LB} to deduce that
$$H^i(U_{\bbC_p},\bbB[1/t]^\varphi)= H^i_{dR}(U_{\bbC_p},\bbQ_p)\otimes B[1/t]^\varphi.$$
Considering the long exact cohomology sequence attached to the short exact sequence (\ref{ses}) we get the claim.
\end{proof}

\begin{Remark}
It seems that this result is also covered by \cite[Theorem 1.1 and Theorem 1.3]{CN2}. For open balls this was done before 
by Colmez and Niziol \cite[Theorem 3]{CN1} resp. Le Bras \cite[Theorem 2.3.2]{LB}. 
\end{Remark}

 Hence we my write $E_1^{\bullet,\bullet}$ as an extension
$$0 \to F_1^{r,s} \to E^{r,s}_1 \to G_1^{r,s} \to 0$$
of double complexes, as well.

The cohomology of the double complex $G_1^{r,s}$ gives rise to a $\bbQ_p$-form $H^\ast_{dR}(\bbP^d_K/\bbP^j_K(\epsilon_n),\bbQ_p)$
of  the de Rham cohomology $H^\ast_{dR}(\bbP^d_K/\bbP^j_K(\epsilon_n))$ which is $\bigoplus_{i=0}^j \bbQ_p[-2i]$.
On the other hand, by Corollary \ref{vanishing_D_Y} and by the lines before Remark \ref{rmk_17} the $F_1^{\bullet,s}$-term just computes the cohomology group
$H^\ast(\P^d_K\setminus \P^j_K(\epsilon_n),\Omega^{\dag,s-1}/D^{\dag,s-1})\hat{\otimes}_K \bbC_p(-s).$ 
More precisely, 
$$F_2^{r,s}=H^r(\P^d_K\setminus \P^j_K(\epsilon_n),\Omega^{\dag,s-1}/D^{\dag,s-1})\hat{\otimes}_K \bbC_p(-s)$$ for all $r,s\geq 0.$
The contributions $H^r(\P^d_K\setminus \P^j_K(\epsilon),\Omega^{\dag,s-1}/D^{\dag,s-1})$ vanish (by considering the \u{C}ech complex) 
for $r \geq d-j.$ Further, we have the long exact 
cohomology sequence
\begin{eqnarray*}
\ldots \rightarrow & H^i_{\bbP^j_K(\epsilon_n)}(\P^d_K,\Omega^{\dag,s-1}/D^{\dag,s-1}) & \rightarrow H^i(\P^d_K,\Omega^{\dag,s-1}/D^{\dag,s-1}) \rightarrow H^i(\P^d_K/\P_K^j(\epsilon_n),\Omega^{\dag,s-1}/D^{\dag,s-1})\\ 
\rightarrow & H^{i+1}_{\bbP^j_K(\epsilon_n)}(\P^d_K,\Omega^{\dag,s-1}/D^{\dag,s-1}) & \rightarrow \ldots 
\end{eqnarray*}
The expressions $H^i(\P^d_K,\Omega^{\dag,s-1}/D^{\dag,s-1})$ vanish by Lemma \ref{vanishing_OmodD} for all $i\geq 0.$ Hence we get
$$H^i(\P^d_K\setminus \P^j_K(\epsilon_n),\Omega^{\dag,s-1}/D^{\dag,s-1})=H^{i+1}_{\P^j_K(\epsilon_n)}(\P^d_K,\Omega^{\dag,s-1}/D^{\dag,s-1})$$ for
all $i\geq 0.$
Since $H^{i}_{\P^j_K(\epsilon)}(\P^d_K,\Omega^{\dag,s-1}/D^{\dag,s-1})=0$ for $i<d-j$  by Lemma \ref{coh_vanishing}
we deduce that $F^{r,s}_2=0$ for $r\neq d-j-1$. Hence the $E_2$-term consists of two lines.

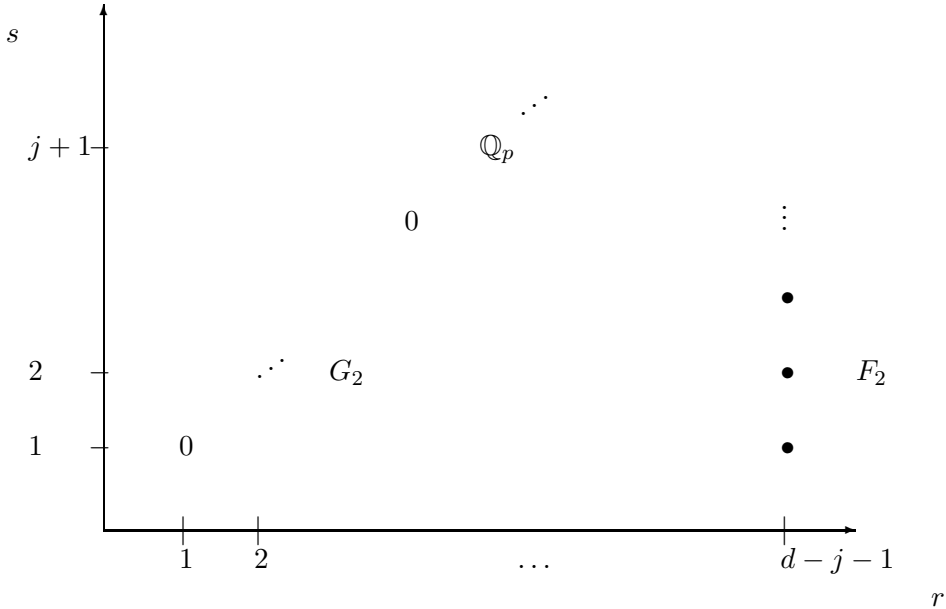
\setlength{\unitlength}{1cm}
\begin{picture}(12,8)
\put(1,0){\vector(0,1){7}}
\put(1,0){\vector(1,0){10}}
\put(-0.3,6.5){$s$}
\put(2,-0.1){$|$}
\put(2,-0.5){$1$}
\put(2,1){$0$}
\put(3,-0.1){$\mid$}
\put(3,-0.5){$2$}
\put(6,5){$\bbQ_p$}
\put(4,2){$G_2$}
\put(3,2){\reflectbox{$\ddots$}}
\put(5,4){$0$}

\put(0.8,1){$-$}
\put(0,1){1}
\put(0.8,2){$-$}
\put(0,2){2}

\put(6.5,-0.5){$\ldots$}
\put(0.8,5){$-$}
\put(0,5){$j+1$}
\put(8,-0.5){}
\put(10,-0.1){$\mid$}
\put(10,-0.5){$d-j-1$}
\put(10,1){$\bullet$}
\put(10,2){$\bullet$}
\put(11,2){$F_2$}
\put(10,3){$\bullet$}
\put(10,4){$\vdots$}

\put(12,-1){$r$}

\put(6.5,5.5){\reflectbox{$\ddots$}}

\end{picture}

\vspace{1cm}

\begin{Lemma}\label{not_include_trivial}
 Considered as $U(\frg)$-module the representations $F_2^{d-j-1,s}, s=0,\ldots d,$ do not include the trivial representation as 
 composition factor.
\end{Lemma}

\begin{proof}
We consider the weights of $F_2^{d-j-1,s}$ with respect to the Cartan algebra $\frt \subset \frg.$ 
By Prop. \ref{letzte_prop_1} the representation 
$F_2^{d-j-1,s}=H^{d-j}_{\P^j_K(\epsilon_n)}(\P^d_K,\Omega^{\dag,s-1}/D^{\dag,s-1})$ is a 
homomorphic image of $H^{d-j}_{\P^j_K(\epsilon_n)}(\P^d_K,\Omega^{\dag,s-1})=H^{d-j}_{\P^j_K(\epsilon_n)}(\P^d_K,\Omega^{s-1})$
which in turn is a quotient of a (Fr\'echet)-completion 
of some representation 
of the shape 
\begin{equation}\label{representation}
 \bigoplus_{k_0,\ldots, k_{j} \geq 0  \atop {k_{j+1},\ldots, k_d \leq 0  \atop {k_0+\cdots + k_d=0}}} 
K\cdot X_0^{k_0}X_1^{k_1}\cdots X_d^{k_d} \otimes  V_{d-j,z_{d-j}\lambda}
\end{equation}
cf. \cite[Prop. 1.4.2, Cor. 1.4.9]{O} for some
irreducible algebraic representation $V=V_{d-j,z_{d-j}\lambda}$ of the Levi subgroup $L_{(j+1,d-j)}$ of $P_{(j+1,d-j)}.$
Here $\lambda$ is the weight defining $\Omega^{s-1}$ in the sense of loc.cit.
It is given by the tuple $(-s+1,1,\cdots,1,0,\ldots,0)\in \Z^{d+1}$ via the identification $X^\ast(T)=\Z^{d+1}$.\footnote{Moreover $i_0=s-1$ in the notation of loc.cit. as $H^k(\P^d_K,\Omega^k)=K$ for all $k\geq 0.$}
Therefore it suffices to check that the latter representation does not contain the trivial representation as composition factor. 
Going to the definition of $V$ in loc.cit., it turns out that that the weights of this representation are given by all 
concatenations of all permutations of the individual arrays $(0,\ldots,0,-s+d-j)$  (of length $j+1$) and 
$(1,\ldots,1,0\ldots,0)$ (of length $d-j$)  for $d-j\leq s-1$ resp. $(0,\ldots,0,-1,\ldots,-1)$  (of length $j+1$) and 
$(d-j+1-s,0,\ldots,0)$ (of length $d-j$) for $d-j> s-1.$
Since the weight of a polynomial $f=X_0^{k_0}X_1^{k_1}\cdots X_d^{k_d}$ is given by $(k_0,\ldots,k_d)$ we see 
that the trivial weight $(0,\ldots,0)$ is not realizable in the representation (\ref{representation}).
\end{proof}

In particular, we see that $E_2=E_\infty$ and that 
$H^{i}(\P^d_{\bbC_p}\setminus \P^j_{\bbC_p}(\epsilon_n),\bbQ_p)$ is an extension
$$0 \to A \to H^{i}(\P^d_{\bbC_p}\setminus \P^j_{\bbC_p}(\epsilon_n),\bbQ_p)\to  B
\to 0 $$
where $$A=H^{d-j-1}(\P^d_K\setminus \P^j_K(\epsilon_n),\Omega^{\dag,s-1}/D^{\dag,s-1})\hat{\otimes}_K \bbC_p(-s)$$
with $s=i-(d-j-1) \geq 1$ and
$$B= H^{i}_{dR}(\P^d_K\setminus \P^j_K(\epsilon_n),\bbQ_p)(-\frac{i}{2}).$$ In particular the term $A$ vanishes for 
$i < d-j$. Further the term $B$  vanishes for $i\geq 2(d-j).$

Thus we have proved by applying the long exact sequence (\ref{long_exact}) the following statement.
\begin{Proposition}\label{inter_result} Let $j\geq 0.$ 

a) If $i< d-j$, then $H^{i}_{\P^j_{\bbC_p}(\epsilon_n)}(\P^d_{\bbC_p},\bbQ_p)= 0$.
 

b)  If $i\geq d-j$, then $H^{i}_{\P^j_{\bbC_p}(\epsilon_n)}(\P^d_{\bbC_p},\bbQ_p)$ is an extension 
$$ 0 \to C \to  H^{i}_{\P_{\bbC_p}^j(\epsilon_n)}(\P^d_{K},\bbQ_p) \to D \to 0$$
where
$$C=H^{d-j}_{\P_{K}^j(\epsilon_n)}(\P^d_{K},\Omega^{\dag,i-(d-j)-1}/D^{\dag,(i-(d-j)-1})  \hat{\otimes}_K \bbC_p (-(i-(d-j))$$ and 
$$D=H^i_{dR,\bbP^j_K(\epsilon_n)}(\P^d_K,\bbQ_p) (-\frac{i}{2})$$
denotes the "local de Rham cohomology" of $\P^d_K$ with support in $\bbP^j_K(\epsilon_n)$.
\end{Proposition}

\medskip
\section{The proof of the main theorem}
Let $\cY_{\bbC_p}$ be the set-theoretical complement of $\cX_{\bbC_p}$ in $\bbP^d_{\bbC_p}.$
Consider the  topological exact sequence
of locally convex $\bbQ_p$-vector spaces with continuous $G$-action
\begin{eqnarray}\label{long_exact_2}
\nonumber \ldots \rightarrow & H^i_{\cY_{\bbC_p}}(\P^d_{\bbC_p},\bbQ_p) & \rightarrow H^i(\P^d_{\bbC_p},\bbQ_p) \rightarrow H^i(\cX_{\bbC_p},\bbQ_p)\\ 
\rightarrow & H^{i+1}_{\cY_{\bbC_p}}(\P^d_{\bbC_p},\bbQ_p) & \rightarrow \ldots 
\end{eqnarray}

The pro-\'etale cohomology groups of $\bbP^d_{\bbC_p}$ look like as in the classical setting by the quasi- compactness of projective space, i.e.
$H^\ast(\P^d_{\bbC_p},\bbQ_p)=\bigoplus_{i=0}^d \bbQ_p [-2i].$ 
Hence it suffices to understand 
the objects $H^i_{\cY_{\bbC_p}}(\P^d_{\bbC_p},\bbQ_p)$ and the maps $H^i_{\cY_{\bbC_p}}(\P^d_{\bbC_p},\bbQ_p)  \rightarrow H^i(\P^d_{\bbC_p},\bbQ_p) $. For this we   recall the construction \cite{O} of an acyclic resolution of the 
constant sheaf $\Z$ on $\cY_{\bbC_p}$ considered as an object in the category of pseudo-adic spaces \cite{H}.

Let $L$ be again one of our fields $K$ or $\bbC_p$. Set
$${\cY}^{ad}_L:=(\P^d_L)^{ad}\setminus \cX^{ad}_L.$$
This is a closed pseudo-adic subspace of $(\P^d_L)^{ad}.$
For any subset $I\subset \Delta$ with  $\Delta\setminus I=\{\alpha_{i_1} < \ldots < \alpha_{i_r}\},$ set 
$$j(I):=i_1 \mbox { and } Y_{I,L}=\P^{j(I)}_L.$$
Furthermore, let $P_I\subset G$ be the standard parabolic subgroup attached to $I$. Hence the group $P_I$ stabilizes $Y_{I,L}.$
We obtain
\begin{eqnarray}\label{Yad=}
{\cY}_L^{ad}=\bigcup_{I \subset \Delta} \bigcup_{g\in G/P_I}g\cdot Y_{I.L}^{ad} = \bigcup_{g\in G}g\cdot Y_{\Delta \setminus \{\alpha_{d-1}\},L}^{ad}.
\end{eqnarray}

\noindent For any compact open subset $ W\subset G/P_I,$ we set
$$Z_{I,L}^W:= \bigcup_{g\in W} gY_{I,L}^{ad}.$$
Thus
$${\cY}_L^{ad}=\bigcup_{I\subset \Delta \atop |\Delta\setminus I|=1}Z_I^{G/P_I}=Z_{\Delta\setminus \{\alpha_{d-1}\},L}^{G/P_{\Delta\setminus{\{\alpha_{d-1} \}}}}.$$
We consider the natural closed embeddings of pseudo-adic spaces
$$\Phi_{g,I} :gY_{I,L}^{ad} \longrightarrow {\cY}_L^{ad} \mbox { and } \Psi_{I,W} :Z^W_{I,L} \longrightarrow {\cY}_L^{ad}.$$
Put
$$ \Z_{g,I}:=(\Phi_{g,I})_\ast(\Phi_{g,I}^\ast\, \Z) \mbox{ and } \Z_{Z_{I,L}^W}:=(\Psi_{I,W})_\ast(\Psi_{I,W}^\ast\, \Z)$$
Let
${\mathcal C}_I$ be the category of compact open disjoint coverings of $G/P_I$
where the morphisms are given by the refinement-order.
For a covering $c =(W_j)_{j} \in {\mathcal C}_I,$ we denote by $\Z_c$ the sheaf  on $\cY^{ad}_L$ defined by
the image of the natural morphism of sheaves
$$\bigoplus_{W_j \in {\mathcal C}_I}\; \Z_{Z_{I,L}^{W_j}} \hookrightarrow \prod_{g\in G/P_I} \Z_{g,I}.$$
We put
\begin{eqnarray}\label{prod'=limInd}
\sideset{}{'}\prod_{g\in G/P_I} \Z_{g,I} = \varinjlim_{c\in {\mathcal C}_I} \Z_c,
\end{eqnarray}
We obtain the following  complex $C^\bullet_L$ of sheaves on ${\cY}_L^{ad},$

\begin{eqnarray}
\nonumber 0 \rightarrow \Z \rightarrow\!\!\! \bigoplus_{I \subset
\Delta \atop |\Delta\setminus I|=1} \sideset{}{'}\prod_{g\in G/P_I} \Z_{g,I}
\rightarrow \!\!\!\bigoplus_{I \subset \Delta \atop |\Delta\setminus I|=2}
\sideset{}{'}\prod_{g\in G/P_I} \Z_{g,I} \rightarrow \cdots \rightarrow \!\!\!\bigoplus_{I \subset \Delta \atop |\Delta\setminus I|=i}
\sideset{}{'}\prod_{g\in G/P_I} \Z_{g,I} \rightarrow \cdots \\ \label{complex} \\ \nonumber \cdots\rightarrow
\!\!\!\bigoplus_{I \subset \Delta \atop |\Delta\setminus I|=d-1} \sideset{}{'}\prod_{g\in G/P_I} \Z_{g,I}
\rightarrow \sideset{}{'}\prod_{g\in G/P_\emptyset} \Z_{g,\emptyset} \rightarrow 0.
\end{eqnarray}

Concerning the next statement see \cite[Thm 2.1.1]{O}.

\begin{Theorem}\label{Theorem1}
The complex $C^\bullet_L$  is acyclic. \qed
\end{Theorem}

We consider the morphism of topoi $\nu: (\widetilde{\P_{\bbC_p}^d)^{ad}_{\mbox{\tiny pro\'et}}} \to \widetilde{(\P_{\bbC_p}^d)^{ad}_{\mbox{\tiny \'et}}}$. By pulling back the
complex $i_\ast(C_{\bbC_p}^\bullet)$ to  $(\P_{\bbC_p}^d)^{ad}_{\mbox{\tiny pro\'et}}$ where 
$i:\cY_{\bbC_p}^{ad} \hookrightarrow (\P^d_{\bbC_p})^{ad}$ is the inclusion, we get a resolution of the pro-\'etale sheaf 
$i_\ast(\bbZ_{\cY_{\bbC_p}^{ad}})$ on $(\P_{\bbC_p}^d)^{ad}$ since  $\nu^\ast$ is an exact functor.
We denote this complex for simplicity by the same symbols. In fact, we could have defined this complex directly on
the pro-\'etale site  as the  sheaves $\bbZ$ are constant.
In this section we evaluate the spectral sequence which is induced
by the complex (\ref{complex}) applied to
${\rm Ext}_{\widetilde{(\P_{\bbC_p}^d)^{ad}_{\mbox{\tiny pro\'et}}}}^\ast(i_\ast(-),\bbQ_p).$ In the following we also simply write
${\rm Ext}^i(\cdot,\cdot)$ for the $i^{th}$ Ext group in the category ${\widetilde{(\P_{\bbC_p}^d)^{ad}_{\mbox{\tiny pro\'et}}}}.$

As usual there is the identification
$${\rm Ext}^\ast(i_\ast(\Z_{{\cY}_{\bbC_p}^{ad}}),\bbQ_p) =
H^\ast_{{\cY}^{ad}_{\bbC_p}}(\P^d_{\bbC_p},\bbQ_p).$$  Further, we have $H^\ast_{{\cY}^{ad}_{\bbC_p}}(\P^d_{\bbC_p},\bbQ_p) =
H^\ast_{{\cY_{\bbC_p}}}(\P^d_{\bbC_p},\bbQ_p)$ by the very definition of pro-\'etale cohomology.

\begin{Proposition}\label{Ext=projlim} For all subsets  $I\subset \Delta,$ there is an isomorphism
$${\rm Ext}^\ast(i_\ast(\sideset{}{'}\prod\limits_{g\in G/P_I} \Z_{g,I}), \bbQ_p)
= \varprojlim_{n\in \N}\bigoplus_{g\in G_0/P_I^n} H_{gY_{I,\bbC_p}(\epsilon_n)}^{\ast}(\P^d_{\bbC_p}, \bbQ_p).$$
\end{Proposition}

\proof Consider the  family
$$\Big\{gP_I^n\mid g\in G_0, n\in \N\Big\}$$
of compact open subsets in $G/P_I$ which yields cofinal coverings in $\cC_I.$
We obtain by (\ref{prod'=limInd}) the identity
$$\sideset{}{'}\prod\limits_{g\in G/P_I} \Z_{g,I}
=\varinjlim_{c\in {\mathcal C}_I} \Z_c = \varinjlim_{n\in \N} \bigoplus_{g\in G_0/P^n_I} \Z_{Z^{gP^n_I}_{I,\bbC_p}}.$$
Choose an injective resolution $\cI^\bullet$ of the sheaf $\bbQ_p$. 
We get
\begin{eqnarray*}
{\rm Ext}^i(i_\ast(\sideset{}{'}\prod\limits_{g\in G/P_I} \Z_{g,I}),\bbQ_p) &=& H^i({\rm Hom}(i_\ast(\sideset{}{'}\prod\limits_{g\in G/P_I} \Z_{g,I}), \cI^\bullet)) \\
= H^i({\rm Hom}(\varinjlim_{n\in \N} \bigoplus_{g\in G_0/P^n_I} i_\ast(\Z_{Z^{gP^n_I}_{I,\bbC_p}}), \cI^\bullet)) &  = & H^i(\varprojlim_{n\in \N} \bigoplus_{g\in G_0/P^n_I} {\rm Hom}(i_\ast(\Z_{Z^{gP^n_I}_{I,\bbC_p}}), \cI^\bullet)) \\
= H^i(\varprojlim_{n\in \N} \bigoplus_{g\in G_0/P^n_I} H^0_{Z^{gP^n_I}_{I,\bbC_p}}(\P^d_{\bbC_p},\cI^\bullet)) &  &  .
\end{eqnarray*}
We make use of the following lemma. Here $\varprojlim\nolimits_{n\in \N}^{(r)}$ is the r-th right derived functor of $\varprojlim\nolimits_{n\in \N}$.
\begin{Lemma}
Let  $\cI$  be an injective sheaf on the pro\'etale site of $(\P^d_{\bbC_p})^{ad}.$ Then 
$$\varprojlim\nolimits_{n\in \N}^{(r)} \bigoplus_{g\in G_0/P^n_I} H^0_{Z^{gP^n_I}_{I,\bbC_p}}(\P^d_{\bbC_p},\cI)=0 \mbox{ for } r\geq 1. $$
\end{Lemma}
\begin{proof}
The proof works in the same way as in \cite[Lemma 2.2.2]{O}.
\end{proof}
\noindent Thus we get by applying  a spectral sequence argument (note that $\varprojlim^{(r)}=0 $ for $r\geq 2$ \cite{Je}) short exact sequences, $i\in \N,$
$$0\rightarrow {\varprojlim_{n}}^{(1)}\!\!\!\!\! \bigoplus_{g\in G_0/P^n_I} \!\!\!\!\!  
H^{i-1}_{Z^{gP^n_I}_{I,\bbC_p}}(\P^d_{\bbC_p},\bbQ_p) \rightarrow  {\rm Ext}^i(i_\ast(\sideset{}{'}\prod\limits_{g\in G/P_I} \Z_{g,I}),\bbQ_p)
\!\rightarrow \varprojlim_{n} \!\!\bigoplus_{g\in G_0/P^n_I} \!\!\!\!\!  H^i_{Z^{gP^n_I}_{I,\bbC_p}}(\P^d_{\bbC_p},\bbQ_p) \rightarrow 0 .$$

\begin{Lemma} The projective system
$\Big(\bigoplus_{g\in G_0/P^n_I} H^{i-1}_{Z^{gP^n_I}_{I,\bbC_p}}(\P^d_{\bbC_p},\bbQ_p)\Big)_{n\in \N}$ consists of $\bbQ_p$-Fr\'echet spaces and satisfies the
(topological) Mittag-Leffler property for all $i\geq 1$.
\end{Lemma}
\begin{proof}
The proof works in the same way as in \cite[Lemma 2.2.3]{O}. Additionally one replaces the Zariski local cohomology groups 
in loc.cit by the extensions in Proposition \ref{inter_result} and considers the corresponding LHSs and RHSs separately. Whereas 
the situation of the RHSs is trivial the LHSs are treated in the same as in loc.cit. 
\end{proof}

\noindent We deduce  from  \cite{EGAIII} 13.2.4 that 
$$\varprojlim_{n\in \N}\nolimits^{(1)}  \Big(\bigoplus_{g\in G_0/P^n_I} H^{i-1}_{Z^{gP^n_I}_{I,\bbC_p}}(\P^d_{\bbC_p},\bbQ_p)\Big)_{n\in \N} = 0.$$
We obtain the identity
$${\rm Ext}^i(i_\ast(\sideset{}{'}\prod\limits_{g\in G/P_I} \Z_{g,I}),\bbQ_p)\, \cong \, \varprojlim_{n\in \N} \bigoplus_{g\in G_0/P^n_I} H^i_{Z^{gP^n_I}_{I,\bbC_p}}(\P^d_{\bbC_p}, \bbQ_p) . $$
On the other hand, we have  
$$\bigcap_{n\in \N} Z_{I,\bbC_p}^{P_I^n} = \bigcap_{n\in \N} Y_{I,\bbC_p}(\epsilon_n)^{ad} = Y_{I,\bbC_p}^{ad}.$$ Again, by applying 
\cite[Proposition 1.3.3]{O}\footnote{The assumption that the sheaf is coherent is not needed here},  we deduce the identity 
\begin{eqnarray*}
\varprojlim_{n}  H^\ast_{Z^{P^n_I}_{I,\bbC_p}}(\P^d_{\bbC_p}, \bbQ_p) & = & \varprojlim_{n}  H^\ast_{Y_{I,\bbC_p}(\epsilon_n)}(\P^d_{\bbC_p}, \bbQ_p). \\
\end{eqnarray*}
We get
\begin{eqnarray*}
\varprojlim_{n} \bigoplus_{g\in G_0/P^n_I} H^\ast_{Z^{gP^n_I}_{I,\bbC_p}}(\P^d_{\bbC_p}, \bbQ_p)  & = &  \varprojlim_{n}\bigoplus_{g\in G_0/P_I^n} H_{gY_{I,\bbC_p}(\epsilon_n)}^\ast(\P^d_{\bbC_p}, \bbQ_p) \\
\end{eqnarray*}
\noindent Thus the statement of our proposition is proved.
\qed

We analyze now the spectral sequence
\begin{multline}\label{ss}
E_1^{-p,q} =  {\rm Ext}^q(\!\!\!\!\!\bigoplus_{I \subset \Delta
\atop |\Delta\setminus I|=p+1}\!\!\!\!\! i_\ast(\sideset{}{'}\prod_{g\in
G/P_I}\limits \Z_{g,I}),\bbQ_p)  \Rightarrow {\rm
Ext}^{-p+q}(i_\ast(\Z_{{\cY}_{\bbC_p}^{ad}}), \bbQ_p)=
H^{-p+q}_{{\cY}_{\bbC_p}^{ad}}(\P_{\bbC_p}^d, \bbQ_p)
\end{multline}
induced by the acyclic complex (\ref{complex}) in Theorem \ref{Theorem1}. By applying  Proposition \ref{inter_result}   the term 
$H_{gY_{I,\bbC_p}(\epsilon_n)}^q(\P^d_{\bbC_p}, \bbQ_p) $ which appears in $E_1^{-p,q}$ as a direct summand is 
for $q\geq 2(d-j(I))$ a
an extension
$$ 0 \to F \to H_{gY_{I,\bbC_p}(\epsilon_n)}^q(\P^d_{\bbC_p}, \bbQ_p) \to G \to 0$$
where 
$$F=H^{d-j(I)}_{g\P_K^{j(I)}(\epsilon_n)}(\P^d_K,\Omega^{\dag,q-(d-j(I))-1}/D^{\dag,(q-(d-j(I))-1})  \hat{\otimes}_K \bbC_p(q-(d-j(I)))$$
and 
$$G=H^q_{dR}(\P^d_K,\bbQ_p)(-\frac{q}{2}).$$
It is equal to
$$H^{d-j(I)}_{g\P_K^{j(I)}(\epsilon_n)}(\P^d_K,\Omega^{\dag,q-(d-j(I))-1}/D^{\dag,(q-(d-j(I))-1})  \hat{\otimes}_K \bbC_p(q-(d-j(I)))$$
for $2(d-j(I))>q\geq d-j(I).$
 Hence we may write $E_1^{\bullet,\bullet}$ as an extension of  double complexes
$$0 \to F_1^{\bullet,\bullet} \to E_1^ {\bullet,\bullet} \to  G_1^{\bullet,\bullet} \to 0.$$
The $F_1$-term splits as a direct sum $F_1=\bigoplus_{s=1}^d F_{1,s}$ where $F_{1,s}$ is the sub double complex
with fixed Tate twist $s=q-(d-j(I)),$ i.e.,
$$F_{1,s}^{p,q}= \varprojlim_n \bigoplus_{I \subset \Delta \atop |\Delta\setminus I|=p+1} \bigoplus_{g\in G_0/P^n_I} 
H^{q-s}_{g\P_K^{d-(q-s)}(\epsilon_n)}(\P^d_K,\Omega^{\dag,s-1}/D^{\dag,s-1})\hat{\otimes}_K \bbC_p(s).$$
Up to the tensor factor $\hat{\otimes}_K \bbC_p(s)$ the object $F_{1,s}$ is the $E_1$-term of a spectral sequence
considered in \cite{O} with respect to the equivariant sheaf $\Omega^{\dag,s-1}/D^{\dag,s-1}$. The computation 
in \cite[p. 633]{O} shows that the only non-vanishing entries $F_{2,s}^{p,q}$ are given by the tuples $(p,q)=(-j+1,d-s+j), 
j=1,\ldots,d$ and that $H^1_{\cY}(\bbP_K^d,\Omega^{\dag,s-1}/D^{\dag,s-1})\hat{\otimes}_K \bbC_p(s)$ is a successive extension of these non-vanishing
objects.

Concerning the  double complex $G_1^{\bullet,\bullet}$ there are the following (non-trivial) rows
\begin{eqnarray*}
G_1^{\bullet,2d}  &   : ( &  \varprojlim_{n}\bigoplus_{g\in G_0/P_\emptyset^n} \bbQ_p \leftarrow 
\bigoplus_{I\subset \Delta \atop { \#I=1 }} \varprojlim_{n}\bigoplus_{g \in G_0/P_I^n} \bbQ_p
\leftarrow \bigoplus_{I\subset \Delta \atop { \#I=2}} \varprojlim_{n}\bigoplus_{g \in G_0/P_I^n} \bbQ_p 
\\ & & \leftarrow  \ldots \leftarrow \bigoplus_{I\subset \Delta \atop { \#I=d-1}} \varprojlim_{n}
\bigoplus_{g \in G_0/P_I^n} \bbQ_p)(-d)\\
\end{eqnarray*}
\begin{eqnarray*} 
G_1^{\bullet,2(d-1)} & : ( & \varprojlim_{n} \bigoplus_{g\in G_0/P_{(2,1,\ldots,1)}^n}\!\!\!\!\!\!\! \bbQ_p \leftarrow 
\bigoplus_{I\subset \Delta \atop {\#I=2 \atop \alpha_0\in I }} \varprojlim_{n}\bigoplus_{g \in G_0/P_I^n} \bbQ_p
\leftarrow \bigoplus_{I\subset \Delta \atop { \#I=3 \atop \alpha_0\in I}} \varprojlim_{n}\bigoplus_{g \in G_0/P_I^n}\bbQ_p \\ & & 
\leftarrow  \ldots  \leftarrow \bigoplus_{I\subset \Delta \atop { \#I=d-1  \atop \alpha_0\in I }} 
\varprojlim_{n}\bigoplus_{g \in G_0/P_I^n} \bbQ_p)(-(d-1))\\
\end{eqnarray*}
\begin{eqnarray*}
& & \centerline{$\vdots$}\\ \\
G_1^{\bullet,2j} & : ( & \varprojlim_{n} \bigoplus_{g\in G_0/P_{(d+1-j,1,\ldots,1)}^n} 
\!\!\!\!\!\!\! \!\!\!\!\!\!\!\bbQ_p\leftarrow 
\bigoplus_{I\subset \Delta \atop { \#I=d-j+1 \atop \alpha_0,\ldots, \alpha_{d-j-1}\in I}}
\varprojlim_{n}\bigoplus_{g \in G_0/P_I^n} \bbQ_p \leftarrow 
\bigoplus_{I\subset \Delta \atop { \#I=d-j+2 \atop \alpha_0,\ldots, \alpha_{d-j-1} \in I}}
\varprojlim_{n}\bigoplus_{g \in G_0/P_I^n} \bbQ_p \\ & & \leftarrow \ldots  \leftarrow 
\bigoplus_{I\subset \Delta \atop { \#I=d-1 \atop \alpha_0,\ldots, \alpha_{d-j-1}\in I}}
\varprojlim_{n} \bigoplus_{g \in G_0/P_I^n}  \bbQ_p)(-j)\\
\end{eqnarray*}
\begin{eqnarray*}\\ 
& & \centerline{$\vdots$}  \\
G_1^{0,2} &: &  \varprojlim_{n}\bigoplus_{g\in G_0/P_{(d,1)}^n} \bbQ_p(-1).
\end{eqnarray*}
\noindent Here, the very left term in each row $G^{\bullet,2j}_1$ sits in degree $-j+1.$
We can rewrite these complexes in terms of induced representations. Here we abbreviate 
$$(d+1-j,1^j):=(d+1-j,1,\ldots,1)$$
for any decomposition $(d+1-j,1,\ldots,1)$ of $d+1.$
Let ${\rm Ind}^{\infty,G}_P$ denote the (unnormalized) smooth
induction functor for a parabolic subgroup $P \subset G$. The dual of the row $G_1^{\bullet,2j}$
coincides with the complex
$${\rm Ind}^{\infty,G}_{P_{(d+1-j,1^j)}} \bbQ_p \rightarrow \!\!\!\!\!\!\!\!\!\! \bigoplus_{I\subset \Delta \atop { \#I=d-j+1 \atop \alpha_0,\ldots, \alpha_{d-j-1}\in I }}\!\!\!\!\!\!\!\!\!\! {\rm Ind}^{\infty, G}_{P_I} \bbQ_p
\rightarrow
\ldots  \rightarrow \bigoplus_{I\subset \Delta \atop {\#(\Delta\setminus I)=1 \atop \alpha_0,\ldots ,\alpha_{d-j-1} \in I}} {\rm Ind}^{\infty,G}_{P_I} \bbQ_p .$$
\medskip
Each of the complexes $G^{\bullet,2j}_1, \; j=1,\ldots,d$,  is acyclic apart 
from the very left and right position \cite[Thm. 7.1.9]{DOR}. Indeed, let
$$v^G_{P_{(d+1-j,1^j)}}(\bbQ_p):= {\rm Ind}^{\infty,G}_{P_{(d+1-j,1^j)}}\bbQ_p/\sum_{Q\supsetneq P_{(d+1-j,1^j)} } {\rm Ind}^{\infty,G}_{Q} \bbQ_p$$ be the smooth generalized Steinberg representation  with respect to the parabolic subgroup  $P_{(d+1-j,1^j)}.$ This is an irreducible smooth $G$-representation, cf.\, \cite[ch. X]{BW}.  We deduce that 
the only non-vanishing entries in $G_2^{p,q}$ are given by the indices $(p,q)=(-j+1,j),\; j=1,\ldots,d,$ and $(p,q)=(0,2j),\; j\geq 2.$
Here we get for $j\geq 2$, $$G_2^{-j,j+1}=v^G_{P_{(d+1-j,1^j)}}(\bbQ_p)(-j)'  \mbox{ and } G_2^{0,2j}=\bbQ_p(-j)$$ and 
$$G_2^{0,2}=({\rm Ind}^{\infty,G}_{P_{(d,1)}} \bbQ_p)'(-1). $$

\medskip
Considered as $U(\frt)$-modules the objects in $G_2$ consist of copies of the trivial representation. Again as in Lemma \ref{not_include_trivial} the contributions of $F_2$ do not contain 
the trivial representation. Hence there are no non-trivial homomorphisms between $G$ and 
$F$ and  we get in particular $E_2=E_\infty$.

For any integer $s\geq 1$, let $V_s^\bullet= V_s^{-d+1} \supset V_s^{-d+2} \supset \cdots\supset  V_s^{-1} \supset V^0_s \supset V^1_s= (0)$ be the canonical filtration on 
$V_s^{-d+1}=H^1_{\cY}(\P^d_K,\Omega^{\dag,s-1}/D^{\dag,s-1})$  defined by the spectral sequence
$$
E_{1,s}^{-p,q} =  {\rm Ext}^q(\!\!\!\!\!\bigoplus_{I \subset \Delta
\atop |\Delta\setminus I|=p+1}\!\!\!\!\! i_\ast(\sideset{}{'}\prod_{g\in
G/P_I}\limits \Z_{g,I}),\Omega^{\dag,s-1}/D^{\dag,s-1})  \Rightarrow H^{-p+q}_{{\cY}^{ad}}(\P_{\bbC_p}^d, \Omega^{\dag,s-1}/D^{\dag,s-1})
$$
induced by the acyclic complex (\ref{complex}) in Theorem \ref{Theorem1}.
In \cite[Lemma 2.7]{O} we proved that in the case of vector bundles that these subspaces  are closed. On the other hand,  
we proved \cite{O2} that the filtration behaves functorial with respect to morphisms of 
vector bundles\footnote{In loc.cit. we stated the lemma for morphisms of vector bundles. But as one see without any difficulties  from the proof this fact is true for arbitrary sheaf morphisms}. 
Applying this fact to the (continuous) morphism $d^s:\Omega^s \to \Omega^{s+1}$ we  deduce that  the subspaces $V^j_s$ are closed in the $K$-Fr\'echet space $H^1_{\cY}(\P^d_K,\Omega^{\dag,s-1}/D^{\dag,s-1})$.

Consider finally  for $\cF^\dag=\Omega^{\dag,s-1}/D^{\dag,s-1}$ the topological exact $G$-equivariant sequence of $K$-Fr\'echet spaces
$$\ldots \rightarrow H^0(\P^d_K,\cF^\dag) \rightarrow H^0(\cX,\cF^\dag)\stackrel{p}{\rightarrow} H^{1}_{\cY}(\P^d_K,\cF^\dag) 
\rightarrow  H^{1}(\P^d_K,\cF^\dag)\rightarrow \ldots.$$
By Proposition \ref{vanishing_OmodD} the terms $H^i(\P^d,\cF^\dag)=0$ vanish for $i\geq 0.$ Hence $p$ is an isomorphism.
For $i=0,\ldots,d,$ we set
$$Z^i_s:=p^{-1}(V^{-d+i+1}_s).$$
Thus we get a $G$-equivariant filtration by closed $K$-Fr\'echet spaces 
$$Z^{0}_s \supset Z^{1}_s \supset \cdots \supset Z_s^{d-1} \supset Z_s^d$$ on $Z_0^{d}=H^0(\cX,\cF^\dag).$

Summarizing the evaluation  of the spectral sequence  we obtain the following theorems
mentioned in the introduction.

\begin{Theorem}
For $j=1,\ldots,d,$ the $p$-adic  pro-\'etale cohomology groups of $\cX_{\bbC_p}$ are extensions of continuous  $G\times \Gamma_K$-represen\-tations 
$$0 \rightarrow (\Omega^{\dag,s-1}/D^{\dag,s-1})(\cX)\hat{\otimes}_K \bbC_p(-s) \rightarrow H^s(\cX_{\bbC_p},\bbQ_p) \rightarrow 
v^{G}_{{P_{(d-s+1,1,\ldots,1)}}}(\bbQ_p)'(-s) \rightarrow 0. $$
\end{Theorem}

\begin{proof}
 In order to complete the proof, we mention that the contributions $G_2^{0,2k},$ $k \geq 2,$ are mapped isomorphically to the cohomology
 groups $H^{2k}(\P^d_{\bbC_p},\bbQ_p)$ in the long exact cohomology sequence (\ref{long_exact_2}). For $k=1,$ we have a 
 surjection $G_2^{0,1} \to \bbQ_p$ whose kernel is isomorphic to $v^G_{P_{(d,1)}}(\bbQ_p)(-1)'.$
\end{proof}

\begin{Proposition}\label{van_coh_drin}
Let $s \geq 0$. Then $H^n(\cX,D^{\dag,s})=0$ for all $n>0.$ 
\end{Proposition}

\begin{proof}
We may compute the local cohomology groups $H^i_{\cY}(\bbP^d_K,D^{\dag,s})$ by the spectral
sequence
$$E_1^{-p,q} =  {\rm Ext}^q(\!\!\!\!\!\bigoplus_{I \subset \Delta
\atop |\Delta\setminus I|=p+1}\!\!\!\!\! i_\ast(\sideset{}{'}\prod_{g\in G/P_I}\limits \Z_{g,I}), D^{\dag,s})  
\Rightarrow H^{-p+q}_{{\cY}^{ad}}(\P_K^d, D^{\dag,s})
$$
induced by the acyclic complex (\ref{complex}) in Theorem \ref{Theorem1}. As we learned in section 1 the computation proceeds in the same way
as for equivariant vector bundles. Here we saw that in loc.cit. (p. 633) that $H^i_{\cY}(\bbP^d_K,D^{\dag,s})=H^i(\bbP^d_K,D^{\dag,s})$ for all $i\geq 2.$
Moreover, for $i=1$ we have an epimorphism $H^i_{\cY}(\bbP^d_K,D^{\dag,s})\to H^1(\P^d_K,D^{\dag,s}).$
By considering the corresponding long exact cohomology sequence we deduce the claim.
\end{proof}

\begin{Remark} \rm
It follows that $H^0(\cX,\Omega^{\dag,s}/D^{\dag,s})=\Omega^{\dag,s}(\cX)/D^{\dag,s}(\cX)$ for any integer $s\geq 0.$ In particular, our result
for the pro-\'etale cohomology of  $\cX$  coincides with the formula in \cite{CDN}. Moreover, the strong
dual $H^0(\cX,\Omega^{\dag,s}/D^{\dag,s})'$ is as a closed subspace of $H^0(\cX,\Omega^s)'$ a locally analytic $G$-representation, as well. 
The same holds henceforth for the quotient $H^0(\cX, D^{\dag,s})'$. In particular, both representations are strongly admissible 
as an extension of strongly admissible representations and in particular admissible \cite{ST2}.
\end{Remark}

Concerning the structure of the term on the left hand side we can make precise it by \cite{O,OS}. Here
we refer to \cite{OS} for the definition of the bi-functors $\cF^G_P$. Here we use Remark \ref{remark_cato} in order to see
that algebraic local cohomology $H^{d-j}_{\bbP^{j}_K}(\bbP^d_K,\Omega^{s-1}/D^{\dag,s-1})$ is an object of the category
$\cO^P$.

\begin{Theorem}
For any fixed integer $s=1,\ldots,d$, there is a descending filtration $(Z_s^i)_{i=0,\ldots,d}$ on 
$Z_s^0=H^0(\cX,\Omega^{\dag,s-1}/D^{\dag,s-1})$ 
by closed subspaces together with isomorphisms of locally analytic $G$-representations
$$(Z_s^i/Z^{i+1}_s)'  \cong \cF^G_{P_{(i+1,d-i)}}(H^{d-i}_{\bbP^{i}_K}(\bbP^d_K,\Omega^{s-1}/D^{s-1}), St_{d-i}),
i=0,\ldots,d-1.$$   where $St_{d-j}$ is the Steinberg representation of ${\rm GL}_{d-j}(K)$ considered as representation of 
$L_{(j+1,d-j)}$.
 \qed
\end{Theorem}

\begin{Remark}
{\rm   In the case of equivariant vector bundles we used in the above  formula 
rather the reduced local cohomology $\tilde{H}^{d-j}_{\P^j_K}(\P^d_K,\cF):= {\rm ker}\,\big(H^{d-j}_{\P^j_K}(\P^d_K,\cF) 
\rightarrow H^{d-j}(\P^d_K,\cF)\big)$. Concerning the compatibility  we note that for 
$\cF=\Omega^{s-1}/D^{s-1}$ we have $\tilde{H}^{d-j}_{\P^j_K}(\P^d_K,\cF)=H^{d-j}_{\P^j_K}(\P^d_K,\cF)$ by Lemma
\ref{vanishing_OmodD}.}
\end{Remark}

\begin{Remark}
The above approach for the determination of the $p$-adic pro-\'etale cohomology works also for the 
$\ell$-adic pro-\'etale cohomology with $\ell \neq p.$ In this case one gets as cited in \cite[Thm. 1.2]{CDN}
$H^s(\cX_{\bbC_p},\bbQ_\ell) =v^{G}_{{P_{(d-s+1,1,\ldots,1)}}}(\bbQ_\ell)'(-s)$ for all $s\geq 0.$
\end{Remark}



\begin{thebibliography}{ABCDEF}

\bibitem[BW]{BW} A. Borel, N.  Wallach, {Continuous cohomology, discrete subgroups, and representations of reductive groups}, Second edition. Mathematical Surveys and Monographs, {\bf 67}, American Mathematical Society, Providence, RI (2000).
\bibitem[Bo]{Bo} G. Bosco, {\it $p$-adic cohomology of Drinfeld spaces}, Master thesis Universit\'e Paris-Sud (2019).
\bibitem[CN1]{CN1} P. Colmez, W. Niziol, {\it On the cohomology of the affine space},  
preprint arXiv: 1707.01133.
\bibitem[CN2]{CN2} P. Colmez, W. Niziol, {\it On $p$-adic comparison theorems for rigid analytic varieties, I}, 
preprint arXiv: 1905.0472. 
\bibitem[CDN]{CDN} P. Colmez, G. Dospinescu, W. Niziol, {\it The pro-\'etale cohomology of $p$-adic Stein spaces}, 
preprint arXiv: 1801.06686.
\bibitem[D]{D} V.G. Drinfeld, {\it Coverings of $p$-adic symmetric regions}, Funct. Anal. and Appl. {\bf 10}, 29 - 40 (1976).
\bibitem[DOR]{DOR} J.-F. Dat, S. Orlik, M. Rapoport, {\it Period domains over finite and p-adic fields},  Cambridge Tracts in Mathematics (No. {\bf 183}) (2010).
\bibitem[EGAIII]{EGAIII} A. Grothendieck, {\it \'Elements de g\'eom\'etrie alg\'ebrique. III, \'Etude cohomologique des faisceaux coh\'erents. I.},  Inst. Hautes \'Etudes Sci. Publ. Math. No. {\bf 11} (1961).
\bibitem[GK]{GK} E. Gro\ss{}e-Kl\"onne, {\it De Rham cohomology of rigid spaces}, Mathematische Zeitschrift {\bf 247}, 223--240 (2004).
\bibitem[GK2]{GK2} E. Gro\ss{}e-Kl\"onne, {\it Rigid analytic spaces with overconvergent structure sheaf}, Journal 
f\"ur die reine und angewandte Mathematik (Crelles Journal) {\bf 519}, 73-95 (2000). 
\bibitem[Ha]{Ha} R. Hartshorne, {\it Algebraic Geometry}, Graduate Texts in Mathematics, No. {\bf 52}, Springer-Verlag, New York-Heidelberg (1977).
\bibitem[H]{H} R. Huber, {\it \'Etale Cohomology of Rigid Analytic
Varieties and Adic spaces}, Aspects of Math., Vol {\bf E 30}, Vieweg (1996).
\bibitem[Je]{Je} C.U. Jensen, {\it Les foncteurs d\'eriv\'es de $\varprojlim$ et leurs applications en th\'eorie des modules}, Lecture Notes in Mathematics, Vol. {\bf 254}, Springer-Verlag, Berlin-New York (1972).
\bibitem[K]{K} R. Kiehl, {\it Theorem A und Theorem B in der nichtarchimedischen Funktionentheorie},  Invent. Math.  {\bf 2}, 256--273 (1967).
\bibitem[LB]{LB} A-C. Le Bras, {\it Anneaux de Fontaine et geometrie: deux exemples d’interaction}. Th\`ese Paris 6 (2017).

\bibitem[O]{O} S. Orlik, {\it Equivariant vector bundles on Drinfeld' upper half space}, Invent. math. {\bf 172}, 
585 - 656 (2008).
\bibitem[O2]{O2} S. Orlik, {\it The de Rham cohomology of Drinfeld's upper halfspace}, M\"unster Journal of Mathematics 8, 
169 - 179 (2015).
\bibitem[OS]{OS} S. Orlik, M. Strauch, {\it On {J}ordan-{H}\"older series of some locally analytic
              representations}, J. Amer. Math. Soc. {\bf 28}, 99--157 (2015).
\bibitem[vP]{vP}  M. van der Put, {\it Cohomology on affinoid spaces}. Compositio Math. 45, no. 2.
\bibitem[ST1]{ST2} P. Schneider, J. Teitelbaum, {\it Algebras of $p$-adic distributions and admissible representations}, Invent. Math. {\bf 153}, No.1, 145-196 (2003).
\bibitem[ST2]{ST3} P. Schneider, J. Teitelbaum, {\it Locally analytic distributions and $p$-adic representation theory, with applications to ${\rm GL}\sb 2$}, J. Amer. Math. Soc.  {\bf 15},  no. 2, 443--468  (2002).
\bibitem[SS]{SS} P. Schneider, U. Stuhler, {\it The cohomology of p-adic
  symmetric spaces}, Invent. math. {\bf 105}, 47-122 (1991).

\end{thebibliography}
\end{document}